\documentclass{article}

\usepackage{amsmath,amssymb}
\usepackage{amsthm,hyperref}
\usepackage{cite}
\usepackage{enumitem}
\usepackage{makecell}
\usepackage{tikz}
\tikzset{
every node/.style={draw, circle, inner sep=2pt}
}

\newtheorem{theorem}{Theorem}[section]
\newtheorem{lemma}[theorem]{Lemma}
\newtheorem{proposition}[theorem]{Proposition}
\newtheorem{corollary}[theorem]{Corollary}

\theoremstyle{definition}
\newtheorem{definition}[theorem]{Definition}

\newtheorem{remark}[theorem]{Remark}
\newtheorem{example}[theorem]{Example}

\newcommand{\bzero}{{\bf 0}}
\newcommand{\trans}{^\top}
\newcommand{\Col}{\mathrm{Col}}
\newcommand{\dunion}{\mathbin{\dot\cup}}
\newcommand{\basisE}{{\mathcal E}}
\newcommand{\basisK}{{\mathcal K}}
\newcommand{\basisX}{{\mathcal X}}

\newcommand{\rank}{\operatorname{rank}}
\newcommand{\vect}{\operatorname{vec}_\triangle}
\newcommand{\vecw}{\operatorname{vec}_\wedge}
\newcommand{\vecs}{\operatorname{vec}_\square}
\newcommand{\supp}{\operatorname{supp}}
\newcommand{\mptn}{\mathcal{S}}
\newcommand{\spec}{\operatorname{spec}}
\newcommand{\mptncl}{\mathcal{S}^{\rm cl}}
\newcommand{\mptnclo}{\mathcal{S}^{\rm cl}_0}

\newcommand{\verS}{\Psi_{\rm SSP}}
\newcommand{\verA}{\Psi_{\rm SAP}}
\newcommand{\mult}{\mathrm{mult}}
\newcommand{\cart}{\mathbin{\square}}
\newcommand{\vspan}{\operatorname{span}}
\newcommand{\tr}{\operatorname{tr}}

\newcommand{\mat}[1][n]{\operatorname{Mat}_{#1}(\mathbb{R})}
\newcommand{\msym}[1][n]{\operatorname{Sym}_{#1}(\mathbb{R})}
\newcommand{\mskew}[1][n]{\operatorname{Skew}_{#1}(\mathbb{R})}

\newcommand{\bx}{{\bf x}}
\newcommand{\bv}{{\bf v}}
\newcommand{\bu}{{\bf u}}
\newcommand{\ba}{{\bf a}}
\newcommand{\bb}{{\bf b}}
\newcommand{\be}{{\bf e}}
\newcommand{\bff}{{\bf f}}
\newcommand{\bs}{{\bf s}}

\usepackage{soul}
\usepackage{cancel}

\newcommand\scalemath[2]{\scalebox{#1}{\mbox{\ensuremath{\displaystyle #2}}}}

\title{The liberation set in the inverse eigenvalue problem of a graph
}
\author{
Jephian C.-H.~Lin
\thanks{Department of Applied Mathematics, National Sun Yat-sen University, Kaohsiung 80424, Taiwan (jephianlin@gmail.com)}
\and
Polona Oblak
\thanks{Faculty of Computer and Information Science, University of Ljubljana, Ve\v cna pot 113, SI-1000 Ljubljana, Slovenia; Faculty of Mathematics and Physics, University of Ljubljana and Institute of Mathematics, Physics, and Mechanics, Jadranska ulica 19, 1000 Ljubljana, Slovenia (polona.oblak@fri.uni-lj.si)}
\and 
Helena \v{S}migoc
\thanks{School of Mathematics and Statistics, University College Dublin, Belfield, Dublin 4, Ireland (helena.smigoc@ucd.ie)}
}
\date{\today}

\begin{document}

\maketitle

\begin{abstract}
The inverse eigenvalue problem of a graph $G$ is the problem of characterizing all lists of eigenvalues of real symmetric matrices whose off-diagonal pattern is prescribed by the adjacencies of $G$. The strong spectral property is a powerful tool in this problem, which identifies matrices whose entries can be perturbed while controlling the pattern and preserving the eigenvalues. The Matrix Liberation Lemma introduced by Barrett et al.~in 2020 advances the notion to a more general setting. In this paper we revisit the Matrix Liberation Lemma and prove an equivalent statement, that reduces some of the technical difficulties in applying the result. 

We test our method on matrices of the form $M=A \oplus B$ and show how this new approach supplements the results that can be obtained from the strong spectral property only. 
While extending this notion to the direct sums of graphs, we discover a surprising connection with the zero forcing game on Cartesian products of graphs. 

Throughout the paper we apply our results to resolve a selection of open cases for the inverse eigenvalue problem of a graph on six vertices.
\end{abstract}  

\noindent{\bf Keywords:} Symmetric matrix; Inverse eigenvalue problem; Strong spectral property; Matrix Liberation Lemma; Zero forcing.

\medskip

\noindent{\bf AMS subject classifications:}
05C50, 
15A18, 
15B57, 
65F18. 

\section{Introduction}
\label{sec:intro}

Let $G$ be a simple graph on $n$ vertices. Let the set $\mptn(G)$ denote the set of $n\times n$ real symmetric matrices whose $(i,j)$-entry, $i\neq j$, is nonzero if and only if $\{i,j\}$ is an edge of $G$, there are no constrains on diagonal entries.  The \emph{inverse eigenvalue problem of a graph} (the IEP-$G$, for short) is to find all possible spectra $\spec(A)$ among matrices $A\in\mptn(G)$.  

Despite huge interest and extensive literature on the problem, the IEP-$G$  has been solved only for a few selected families of graphs that include paths~\cite{MR382314, MR447279, MR447294}, cycles~\cite{MR583498}, generalized stars~\cite{MR2022294}, complete graphs \cite{MR3118942}, lollipop and barbell graphs~\cite{MR4316738}, linear trees \cite{MR4357320}, and graphs with at most five vertices~\cite{MR3291662, MR4074182}. For the background to the problem we refer the reader to~\cite{IEPGZF22}.

The study of the IEP-$G$ involves finding matrices in $\mptn(G)$ with prescribed eigenvalues, or proving that such matrices do not exist. Since those tasks are hard once matrices get large and patterns more complicated, the research focuses on finding ways of extracting information for more difficult cases from simpler or smaller examples. An important advance in this direction was made in  \cite{MR3665573, MR4074182} with the introduction of the strong properties. In particular, if a matrix $A \in \mptn(G)$ has the strong spectral property 
defined below, then we can, for any spanning supergraph $H$ of $G$, infer the existence of a matrix $A' \in \mptn(H)$ that is cospectral with $A$.
The Matrix Liberation Lemma \cite{MR4074182} (see also Lemma~\ref{lemma:liberation}) provides a theoretical foundation for an investigation into when and to what extent similar conclusions can be drawn for matrices that do not have the strong spectral property. A more rigorous background on the strong properties leading to the Matrix Liberation Lemma is detailed in Subsection~\ref{subsec:preliminaries}.

Applying the Matrix Liberation Lemma directly involves first constructing a verification matrix $\Psi$, and then finding a vector $\bf x$ with certain properties in the column space of $\Psi$. Both tasks can be technically demanding, and make it hard to develop an intuition into the process. With the aim of untangling some of the technical difficulties of applying the lemma, we introduce the liberation set of a matrix in Section~\ref{sec:libermtx}. With this notion we are able to form a result that is equivalent to the Matrix Liberation Lemma, but easier to apply.

With this new insight we are able to offer a series of examples that highlight how our result can be used to advance the IEP-$G$. In particular, in Section \ref{sec:directsum} we apply the method to direct sums of matrices. The eigenvalues of $M=A \oplus B$ are straightforward to determine from the eigenvalues of $A$ and $B$. Our method allows us to perturb the matrix $M$ so that the new matrix corresponds to different connected graphs, while preserving the spectrum; see Theorems~\ref{thm:multi-multi-eigenvalue-in-common} and~\ref{thm:multi-1-eigenvalue-in-common}. The strong spectral property already allows this, but only in the case when $A$ and $B$ have no eigenvalues in common.

 Surprisingly, the zero forcing process introduced in \cite{MR2388646} plays an important role in determining the liberation set of the directed sum of matrices with the strong spectral property; see Section~\ref{sec:ZF}.
 In Section~\ref{sec:libergraph} we continue the direction of research initiated in \cite{MR4080669} and define the liberation set of a graph $G$, which is independent of the choice of $A\in\mptn(G)$. 

Throughout the paper we offer a myriad of examples on how our theory can be applied. In particular, in~\cite{MR4284782} a huge advance towards resolving the IEP-$G$ for graphs on six vertices was made, with only spectral arbitrariness of selected multiplicities left to be resolved. We are able to resolve some of those open cases, see Table~\ref{table:6-vertices}. 

\newlength{\Gheight}
\newlength{\struct}
\settoheight{\Gheight}{G}
\setlength{\struct}{0.5\Gheight}
\addtolength{\struct}{-0.6cm}
\begin{table}
\centering
\begin{tabular}{|c|c|l|l|}
\hline
No. & Graph & Ordered multiplicity lists & Reference \\ 
\hline
$\mathsf{G_{100}}$ &
\raisebox{\struct}{\begin{tikzpicture}[scale=0.5, every node/.append style={inner sep=1pt}]
\draw[opacity=0] (0,-1.2) -- (0,1.2);
\foreach \i in {1,...,6} {
    \pgfmathsetmacro{\angle}{60 * (\i - 1)}
    \node (\i) at (\angle:1) {};
}
\draw (3) -- (1) -- (2);
\draw (4) -- (1);
\draw (5) -- (6);
\draw (5) -- (4) -- (6);
\end{tikzpicture}} &
$(1,2,2,1)$ &
Example~\ref{ex:twotrees} \\
\hline 
$\mathsf{G_{127}}$ &
\raisebox{\struct}{\begin{tikzpicture}[scale=0.5, every node/.append style={inner sep=1pt}]
\draw[opacity=0] (0,-1.2) -- (0,1.2);
\foreach \i in {1,...,6} {
    \pgfmathsetmacro{\angle}{60 + 60 * (\i - 1)}
    \node (\i) at (\angle:1) {};
}
\draw (1) -- (2) -- (3) -- (1);
\draw (4) -- (5) -- (6);
\draw (1) -- (6);
\draw (3) -- (4);
\end{tikzpicture}} &
$(2,1,1,2)$ &
Example~\ref{ex:g127g169} \\
\hline 
$\mathsf{G_{129}}$ &
\raisebox{\struct}{\begin{tikzpicture}[scale=0.5, every node/.append style={inner sep=1pt}]
\draw[opacity=0] (0,-1.2) -- (0,1.2);
\foreach \i in {1,4,6} {
    \pgfmathsetmacro{\angle}{60 * (\i -2)}
    \node (\i) at (\angle:1) {};
 }
\foreach \i in {2,3} {
    \pgfmathsetmacro{\angle}{60 * (\i -2)}
    \node (\i) at (\angle:1) {};
 }
\node (5) at (180:1) {};
\draw (1) -- (2) -- (3) -- (4);
\draw (2) -- (6);
\draw (4) -- (5) -- (6);
\draw (1) -- (5);
\end{tikzpicture}} &
$(1,1,3,1)$, $(1,3,1,1)$ &
Example~\ref{ex:g129-g145-g153} \\
\hline 
$\mathsf{G_{145}}$ &
\raisebox{\struct}{\begin{tikzpicture}[scale=0.5, every node/.append style={inner sep=1pt}]
\draw[opacity=0] (0,-1.2) -- (0,1.2);
\foreach \i in {1,4,6} {
    \pgfmathsetmacro{\angle}{60 * (\i -2)}
    \node (\i) at (\angle:1) {};
 }
\foreach \i in {2,3} {
    \pgfmathsetmacro{\angle}{60 * (\i -2)}
    \node (\i) at (\angle:1) {};
 }
\node (5) at (180:1) {};
\draw (1) -- (2) -- (3) -- (4);
\draw (2) -- (6);
\draw (4) -- (5) -- (6);
\draw (1) -- (5)--(2);
\end{tikzpicture}} &
$(1,1,3,1)$, $(1,3,1,1)$ &
Example~\ref{ex:g129-g145-g153} \\
\hline 
$\mathsf{G_{151}}$ &
\raisebox{\struct}{\begin{tikzpicture}[scale=0.5, every node/.append style={inner sep=1pt}]
\draw[opacity=0] (0,-1.2) -- (0,1.2);
\foreach \i in {1,2} {
    \pgfmathsetmacro{\angle}{60 * (\i)}
    \node (\i) at (\angle:1) {};
 }
\node (3) at (0:1) {};
\node (4) at (-60:1) {};
\node (5) at (-120:1) {};
\node (6) at (180:1) {};
\draw (3)--(1) -- (2);
\draw (1) -- (4);
\draw (5) -- (6);
\draw (3) -- (5) -- (4);
\draw (2) -- (6) -- (4);
\end{tikzpicture}} &
\makecell[cl]{$(1,1,3,1)$, $(1,3,1,1)$, $(1,2,3)$, $(3,2,1)$,\\ 
$(1,3,2)$, $(2,3,1)$} &
Example~\ref{ex:g151} \\
\hline 
$\mathsf{G_{153}}$ &
\raisebox{\struct}{\begin{tikzpicture}[scale=0.5, every node/.append style={inner sep=1pt}]
\draw[opacity=0] (0,-1.2) -- (0,1.2);
\foreach \i in {1,4,6} {
    \pgfmathsetmacro{\angle}{60 * (\i -2)}
    \node (\i) at (\angle:1) {};
 }
\foreach \i in {2,3} {
    \pgfmathsetmacro{\angle}{60 * (\i -2)}
    \node (\i) at (\angle:1) {};
 }
\node (5) at (180:1) {};
\draw (6) -- (1) -- (2) -- (3) -- (4);
\draw (2) -- (6);
\draw (4) -- (5) -- (6);
\draw (1) -- (5);
\end{tikzpicture}} &
$(1,1,3,1)$, $(1,3,1,1)$ &
Example~\ref{ex:g129-g145-g153} \\
\hline 
$\mathsf{G_{163}}$ &
\raisebox{\struct}{\begin{tikzpicture}[scale=0.5, every node/.append style={inner sep=1pt}]
\draw[opacity=0] (0,-1.2) -- (0,1.2);
\foreach \i in {1,...,6} {
    \pgfmathsetmacro{\angle}{60 * (\i - 1)}
    \node (\i) at (\angle:1) {};
 }
\draw (1) -- (2) -- (3) -- (1);
\draw (4) -- (5) -- (6) -- (1) -- (5) -- (3) -- (4);
\end{tikzpicture}} &
$(1,1,3,1)$, $(1,3,1,1)$ &
Example~\ref{ex:g163} \\
\hline 
$\mathsf{G_{169}}$ &
\raisebox{\struct}{\begin{tikzpicture}[scale=0.5, every node/.append style={inner sep=1pt}]
\draw[opacity=0] (0,-1.2) -- (0,1.2);
\foreach \i in {1,...,6} {
    \pgfmathsetmacro{\angle}{-60 + 60 * (\i - 1)}
    \node (\i) at (\angle:1) {};
}
\draw (3) -- (1) -- (2) -- (3) -- (4) -- (1);
\draw  (4) -- (5) -- (6) -- (2) -- (4);
\end{tikzpicture}} &
$(1,3,2)$, $(2,3,1)$ &
Example~\ref{ex:g127g169} \\
\hline 
$\mathsf{G_{171}}$ &
\raisebox{\struct}{\begin{tikzpicture}[scale=0.5, every node/.append style={inner sep=1pt}]
\draw[opacity=0] (0,-1.2) -- (0,1.2);
\foreach \i in {1,3,4,5} {
    \pgfmathsetmacro{\angle}{60 * (\i -1)}
    \node (\i) at (\angle:1) {};
 }
\node (2) at (60:1) {};
\node (6) at (300:1) {};
\draw (1) -- (2) -- (3) -- (4)--(5)--(1);
\draw (1) -- (6) -- (3);
\draw (4) -- (6)--(5);
\end{tikzpicture}} &
\makecell[cl]{$(1,1,3,1)$, $(1,3,1,1)$, $(1,2,3)$, $(3,2,1)$, \\
$(1,3,2)$, $(2,3,1)$} &
Example~\ref{ex:g171-g187} \\
\hline 
$\mathsf{G_{175}}$ &
\raisebox{\struct}{\begin{tikzpicture}[scale=0.5, every node/.append style={inner sep=1pt}]
\draw[opacity=0] (0,-1.2) -- (0,1.2);
\foreach \i in {1,3,5} {
    \pgfmathsetmacro{\angle}{60 * (\i -1)}
    \node (\i) at (\angle:1) {};
 }
\foreach \i in {2,4} {
    \pgfmathsetmacro{\angle}{60 * (\i -1)}
    \node (\i) at (\angle:1) {};
 }
\node (6) at (300:1) {};
\draw (1) -- (6) -- (3);
\draw (5) -- (6);
\draw (5) -- (4) -- (1);
\draw (5) -- (2) -- (1);
\draw (4) -- (3) -- (2);
\end{tikzpicture}} &
$(1,3,2)$, $(2,3,1)$ &
Example~\ref{ex:g175} \\
\hline 
$\mathsf{G_{187}}$ &
\raisebox{\struct}{\begin{tikzpicture}[scale=0.5, every node/.append style={inner sep=1pt}]
\draw[opacity=0] (0,-1.2) -- (0,1.2);
\foreach \i in {1,2,3,4,5} {
    \pgfmathsetmacro{\angle}{60 * (\i -1)}
    \node (\i) at (\angle:1) {};
 }
\node (6) at (300:1) {};
\draw (1) -- (2) -- (3) -- (4)--(5)--(1);
\draw (1) -- (6) -- (3);
\draw (4) -- (6)--(5);
\draw (6)--(2);
\end{tikzpicture}} &
\makecell[cl]{$(1,1,3,1)$, $(1,3,1,1)$, $(1,2,3)$, $(3,2,1)$, \\
$(1,3,2)$, $(2,3,1)$} &
Example~\ref{ex:g171-g187} \\
\hline 
\end{tabular}
    \caption{ 
The table of realizable ordered multiplicity lists for graphs on six vertices,~\cite[Appendix B]{MR4284782},  that we prove are spectrally arbitrary in this paper. Graphs are labeled according to the labeling in Atlas of Graphs~\cite{MR1692656}.
    \label{table:6-vertices}
    }
\end{table}

\subsection{Notation and terminology}

As the topic of this work spans matrix and graph theory, we depend on some standard notation from both areas. 

For spaces of matrices, $\mat[m,n]$ denotes the space of all $m\times n$ matrices over $\mathbb{R}$ (where $m = n$ case is abbreviated as $\mat$),  $\msym$ denotes the space of all symmetric matrices of order $n$ over $\mathbb{R}$, and $\mskew$ be the space of all skew-symmetric matrices of order $n$ over $\mathbb{R}$.  

While $\mptn(G)$ is not a subspace of $\msym$, the topological closure of $\mptn(G)$ denoted by  $\mptncl(G)$ is.  That is, $\mptncl(G)$ is the set of matrices whose $(i,j)$-entry is nonzero only when $i=j$ or $\{i,j\}$ is an edge of $G$, and it is a subspace of $\msym$ of dimension $|V(G)| + |E(G)|$.  The set $\mptnclo(G)$ of matrices in $\mptncl(G)$ with zero diagonal is another related subspace, this one of dimension $|E(G)|$.  

For any positive integers $n$ and $k$, let  $[n]:=\{1,2,\dots,n\}$ and $k + [n] := \{k + 1, \ldots, k + n\}$. For any given vector $\bv=\begin{pmatrix} v_i\end{pmatrix} \in \mathbb{R}^n$, we define its support by $\supp(\bv):=\left\{i\colon v_{i} \ne 0\right\}$, which is a subset of $[n]$. 
The column space of  an $m\times n$ matrix $M$ is denoted by 
$\mathrm{Col}(M):=\{M\bx\colon \bx \in \mathbb{R}^n\}$ and its kernel (or null space) by $\ker(M):=\{\bx \in \mathbb{R}^n\colon M \bx=\bf{0}\}$.  For $\alpha \subseteq [m]$ and $\beta \subseteq [n]$, let $M[\alpha,\beta]$ be the submatrix of $M$ induced on rows in $\alpha$ and columns in $\beta$.  Either $\alpha$ or $\beta$ can be replaced by the symbol $:$ as an indication of the whole index set of rows or columns.  When we write $\|M\|$, this can represent any matrix norm; however for concreteness we can take $\|M\| := \sqrt{\tr(M\trans M)}$ throughout.

Let $A$ and $B$ be matrices.  Their direct sum is denoted by $A\oplus B$, $A \circ B$ is the Hadamard (entrywise) product of $A$ and $B$, and $[A,B] := AB - BA$ is the commutator, where in each case we assume $A$ and $B$ are of sizes compatible with the relevant operation.

Let $E^{i,j} \in \mat[m,n]$ be the matrix with one at position $(i,j)$ and zeros elsewhere, $I_n$ the $n \times n$ identity matrix, and $O_{m,n}$ the $m \times n$ zero matrix. In all cases the indices will be omitted if they are clear from the context. Moreover, let $K^{i,j}:=E^{i,j}-E^{j,i}$ and $X^{i,j} := E^{i,j} + E^{j,i}$. 

The \emph{lexicographical order} on $\mathbb{Z}\times\mathbb{Z}$ is defined by $(i,j)\preceq (k,\ell)$ if and only if $i<k$ or ($i=k$ and $j\leq \ell$).  For  $A\in \mat[m,n]$ let $\vect(A)$ be the vector in $\mathbb{R}^{\binom{n+1}{2}}$ that records the entries in the upper triangular part of $A$ under the lexicographical order.  Similarly, $\vecw(A)$ is the vector in $\mathbb{R}^{\binom{n}{2}}$ that records the strictly upper triangular part of $A$, excluding the diagonal entries.  Finally, $\vecs(A)$ is the vector in $\mathbb{R}^{n^2}$ that records all entries of $A$ under the lexicographical order.

For a symmetric matrix $A$ we denote 
the multiplicity of an eigenvalue $\lambda$ of $A$ as  $\mult_A(\lambda)$. Suppose that $A\in \msym$ has distinct eigenvalues $\lambda_1, \ldots,\lambda_{q}$ with multiplicities $m_i=\mult_A(\lambda_i)$ for $i\in[q]$. The spectrum of $A$ will be denoted by $\spec(A)=\{\lambda_1^{(m_1)},\ldots,\lambda_q^{(m_q)}\}$, where $\lambda^{(k)}$ denotes $k$ copies of $\lambda$. The \emph{unordered multiplicity list} of $A$ is defined to be the list of multiplicities $\{m_1,\ldots,m_q\}$, in no particular order. We say that the unordered multiplicity list $\{m_1,\ldots,m_q\}$ is \emph{spectrally arbitrary} for $G$ if for any distinct $\lambda_1, \ldots,\lambda_{q}$, the spectrum $\{\lambda_1^{(m_1)},\ldots,\lambda_q^{(m_q)}\}$ is realizable by a matrix in $\mptn(G)$. 
If the eigenvalues of $A$ are ordered in an increasing order $\lambda_1< \cdots<\lambda_{q}$, then the \emph{ordered multiplicity list} of $A$ is defined as an ordered list $(m_1,\ldots,m_q)$. We say that the ordered multiplicity list $(m_1,\ldots,m_q)$ is \emph{spectrally arbitrary} for $G$ if for any $\lambda_1< \cdots< \lambda_{q}$, the spectrum $\{\lambda_1^{(m_1)},\ldots,\lambda_q^{(m_q)}\}$ is realizable by a matrix in $\mptn(G)$.

The \emph{disjoint union} of two graphs $G$ and $H$ will be denoted by $G \dunion H$, and the complement of graph $G$ by $\overline G$. If $G$ and $H$ are graphs with $V(G) = V(H)$ and $E(H)= E(G) \cup \{e_1,\ldots, e_k\}$, we say that $H$ is a \emph{spanning supergraph} of $G$ and write $H=G+\{e_1,\ldots, e_k\}$.

\subsection{Preliminaries}\label{subsec:preliminaries}

A symmetric matrix $A$ has the \emph{strong spectral property} (the SSP, for short) if $X = O$ is the only symmetric matrix satisfying $A \circ X = I \circ X = O$ and $[A,X] = O$.   The SSP was first introduced in~\cite{MR3665573} and was motivated by the strong Arnold property introduced in \cite{MR1070462, MR1224700}.   A symmetric matrix $A$ has the \emph{strong Arnold property} (the SAP, for short) if $X = O$ is the only symmetric matrix satisfying $A \circ X = I \circ X = O$ and $AX = O$. Having a matrix with a strong property in hand, we can infer the existence of other matrices that share the relevant spectral property, as detailed in the two results below. 

\begin{theorem}[Supergraph Lemma \cite{MR3665573}]
\label{thm:supergraph}
Let $G$ be a graph and $H$ a spanning supergraph of $G$.
If $A\in \mptn(G)$ has the SSP (the SAP, respectively), then for any $\epsilon > 0$ there exists $A'\in \mptn(H)$ with $\|A-A'\|<\epsilon$ and the SSP (the SAP, respectively) such that  $\spec(A)=\spec(A')$ ($\rank(A)=\rank(A')$, respectively).
\end{theorem}

\begin{theorem}[Direct Sum Lemma \cite{MR3665573}]
\label{thm:direct sum}
Let $G$ and $H$ be graphs, and let $A\in \mptn(G)$ and $B\in \mptn(H)$ both have the SSP (the SAP, respectively). Then $A\oplus B$ has the SSP (the SAP, respectively) if and only if $\spec(A)\cap \spec(B)=\emptyset$ ($0\notin\spec(A)\cap\spec(B)$, respectively).
\end{theorem}

\begin{example}\label{ex:simpleSSP}
Let $G = K_n\dunion K_1$ and label the vertices $V(G)=[n+1]$ of $G$ so that $n+1$ is the isolated vertex of $G$. Moreover, let
\[A = \begin{pmatrix}
 M & {\bf 0} \\
 {\bf 0} & \lambda
\end{pmatrix}\in \mptn(G),\]
where $M$ is the $n\times n$ all-ones matrix. To check if $A$ has the SSP, let ${\bf b}\in \mathbb{R}^n$ and
\[X = \begin{pmatrix}
 O_{n} & {\bf b} \\
 {\bf b}\trans & 0
\end{pmatrix}.\]
Note that $X \circ A = X \circ I = O$ and the condition $[A,X] = O$ is equivalent to 
$M{\bf b}=\lambda {\bf b}$.
Therefore, $A$ has the SSP if and only if $\lambda\notin \spec(M)= \{0^{(n-1)}, n\}$. By Theorem~\ref{thm:supergraph} there is a matrix $A'\in\mptn(H)$ with the SSP and spectrum $\{0^{(n-1)}, n, \lambda\}$ for any spanning supergraph $H$ of $G$ and $\lambda\notin\{0,n\}$.
\end{example}

From the example we see that some matrices might not have the SSP. However, if we are willing to provide some more restrictions, e.g., assuming ${\bf b}$ in Example~\ref{ex:simpleSSP} has only one nonzero entry, then the matrix is not so far from having the SSP. This leads to a generalization of the Supergraph Lemma.

\begin{definition}\cite[Definition~3.2]{MR4316738}
\label{def:SSPwrt}
Let $G$ be a spanning subgraph of $H$.  A matrix $A\in\mptn(G)$ has \emph{the SSP with respect to $H$} if $X = O$ is the only symmetric matrix that satisfies $X\in\mptnclo(\overline{H})$ and $[A,X] = O$.   

A matrix $A\in\mptn(G)$ has \emph{the SAP with respect to $H$} if $X = O$ is the only symmetric matrix that satisfies $X\in\mptnclo(\overline{H})$ and $AX = O$.
\end{definition}

\begin{remark}
Note that $A\in\mptn(G)$ has the SSP with respect to $G$ itself if and only if $A$ has the SSP, because
\[\mptnclo(\overline{G}) = \{X\in \msym: A\circ X = I \circ X = O\}.\]
Let $G$, $H$, $H'$ be graphs on the same set of vertices such that $E(G) \subseteq E(H) \subseteq E(H')$.  By Definition~\ref{def:SSPwrt} and the fact $\mptnclo(\overline{H}) \supseteq \mptnclo(\overline{H'})$, $A$ has the SSP with respect to $H$ implies $A$ has the SSP with respect to $H'$.  In particular, $A$ has the SSP implies $A$ has the SSP with respect to any spanning supergraph $H$ of $G$.  The statements in this remark also hold for the SAP. 
\end{remark}

\begin{theorem}{\rm\cite[Theorem 3.4]{MR4316738}}
\label{thm:supergraphH}
Let $G$, $H$, and $H'$ be graphs such that $V(G)=V(H)=V(H')$ and $E(G) \subseteq E(H) \subseteq E(H')$.  
If $A\in \mptn(G)$ has the SSP (the SAP, respectively) with respect to $H$, then for any $\epsilon > 0$ there exists $A'\in \mptncl(H')$ with $\|A-A'\|<\epsilon$, $\spec(A)=\spec(A')$ ($\rank(A)=\rank(A')$, respectively), $A$ has the SSP (the SAP, respectively) with respect to $H'$ and every entry of $A'$ that corresponds to an edge in $E(H')\setminus E(H)$ is nonzero.
\end{theorem}

\begin{example}
 We revisit $G$, $A$, and $X$ given in Example~\ref{ex:simpleSSP}.
Recall, that $A$ does not have the SSP for $\lambda \in \{0,n\}$, so let us choose $\lambda=n$.  Let $H:=G+\{n+1,i\}$ for some $i \in [n]$, and $H'$ any supergraph of $H$ with $V(H')=[n+1]$.  

Assuming the $i$-th entry of $\bb$ is zero, it is an easy computation to observe that the equality $M{\bf b}=n{\bf b}$ implies ${\bf b}={\bf 0}$. This implies $A$ has the SSP with respect to $H$. By Theorem~\ref{thm:supergraphH}, there is a matrix $A'\in\mptncl(H')$ with the SSP with respect to $H'$ such that $\spec(A') = \{0^{(n-1)},n^{(2)}\}$. Moreover, $A'$ has nonzero entries corresponding to $E(H')\setminus E(H)$, and since $A'$ can be chosen arbitrarily close to $A$, we may also assume the nonzero entries of $A$ stay nonzero in $A'$. However, Theorem \ref{thm:supergraphH} does not tell us if the entry of $A'$ corresponding to $\{n+1,i\}$ is zero or not. 
\end{example}

As we have seen in the example above, the drawback of applying Theorem \ref{thm:supergraphH} is that there are entries in the matrix $A'$ that can take any real value, so the graph of $A'$ is not completely defined. To attend to this drawback, we depend on work from~\cite{MR4074182,MR3665573}. In particular, we use the notion of the verification matrix and related results as summarized below.

Note that the subspace $\mptnclo(\overline{G})$
has a basis $\basisX = \{X^{i,j}: \{i,j\}\in E(\overline{G})\}$, and the subspace $\mskew$ has a basis $\basisK = \{K^{i,j}: i,j \in [n],\ i<j\}$.  
Define a linear map
\[
f \colon \mptnclo(\overline{G}) \rightarrow \mskew \text{ with } f(X) = [A,X],
\]
and let $\Psi$ be the matrix representation of $f$ from basis $\basisX$ to basis $\basisK$ so that $[X]_\basisX\trans\Psi = [f(X)]_\basisE\trans$.  (The ordering of rows and columns of $\Psi$ respects lexicographic ordering.)  Here $[X]_\basisX$ and $[f(X)]_\basisK$ denote the corresponding vector representations of $X$ and $f(X)$ in bases $\basisX$ and $\basisE$, respectively, and we use the left multiplication to follow the convention in \cite{MR4074182}.   By definition, one may verify whether $A$ has the SSP by checking if the left kernel of $\Psi$ is trivial.  The matrix $\Psi$ is called the SSP verification matrix and is formally defined in Definition~\ref{def:vermtx}.  Note that $[K]_\basisK = \vecw(K)$ for any skew-symmetric matrix $K$.  

\begin{definition}\cite[Theorem~31]{MR3665573}
\label{def:vermtx}
Let $G$ be a graph on $n$ vertices, $\overline{m} = |E(\overline{G})|$, and $A\in\mptn(G)$.  The \emph{SSP verification matrix} $\verS(A)$ is the $\overline{m}\times\binom{n}{2}$ matrix whose rows are $\vecw([A, X^{i,j}])$ for all pairs $(i,j)$ with $1\leq i < j \leq n$ and $\{i,j\}\in E(\overline{G})$ under the lexicographical order.  

The \emph{SAP verification matrix} $\verA(A)$ is the $\overline{m}\times n^2$ matrix whose rows are $\vecs(AX^{i,j})$ for all pairs $(i,j)$ with $1\leq i < j \leq n$ and $\{i,j\}\in E(\overline{G})$ under the lexicographical order.
\end{definition}

\begin{remark}
\label{rem:vermtx}
It is not hard to see (see also~\cite{MR4074182}) that $A$ has the SSP (the SAP, respectively) if and only if the left kernel of the verification matrix $\verS(A)$ ($\verA(A)$, respectively) is trivial. This is true if and only if the set of rows in $\verS(A)$ ($\verA(A)$, respectively) is linearly independent.  

Indeed, similar arguments work for the extended strong properties. That is, $A$ has the SSP (the SAP, respectively) with respect to $H$ if and only if the set of rows in $\verS(A)$ ($\verA(A)$, respectively) corresponding to $E(\overline{H})$ is linearly independent.
\end{remark}

With the verification matrix, we are able to state the Matrix Liberation Lemma introduced in \cite{MR4074182}.  Here we rephrase the statements in terms of the extended strong properties in Definition~\ref{def:SSPwrt}.

\begin{lemma}[Matrix Liberation Lemma --- vector version \cite{MR4074182}]\label{lemma:liberation}
Let $G$ be a graph, $A\in \mptn(G)$, and $\Psi = \verS(A)$ ($\Psi = \verA(A)$, respectively).  Suppose there is a vector $\bx \in \Col(\Psi)$ such that $A$ has the SSP (the SAP, respectively) with respect to $G+\supp(\bx)$.  Then there exists a matrix $A'\in \mptn(G +\supp(\bx))$ with the SSP (the SAP, respectively) and $\spec(A')=\spec(A)$ ($\rank(A')=\rank(A)$, respectively).
\end{lemma}

\begin{example}
\label{ex:libvec}
Let $G = K_4\dunion K_1$ as in Example~\ref{ex:simpleSSP} for $n=4$.   Then the SSP verification matrix of 
\[A = \begin{pmatrix}
 1 & 1 & 1 & 1 & 0 \\
 1 & 1 & 1 & 1 & 0 \\
 1 & 1 & 1 & 1 & 0 \\
 1 & 1 & 1 & 1 & 0 \\
 0 & 0 & 0 & 0 & 4 
\end{pmatrix}\in\mptn(G)\]
 is  equal to
\[\Psi = \begin{pmatrix}
0 & 0 & 0 & -3 & 0 & 0 & 1 & 0 & 1 & 1 \\
0 & 0 & 0 & 1 & 0 & 0 & -3 & 0 & 1 & 1 \\
0 & 0 & 0 & 1 & 0 & 0 & 1 & 0 & -3 & 1 \\
0 & 0 & 0 & 1 & 0 & 0 & 1 & 0 & 1 & -3
\end{pmatrix},\]
where the rows of $\Psi$ are indexed by the nonedges $\{1,5\}$, $\{2,5\}$, $\{3,5\}$, and $\{4,5\}$.  Let $\alpha = \{\{1,5\},\{2,5\}\}$ and $\beta = \{\{3,5\},\{4,5\}\}$.  The vector $\bx= (0,0,1,-1)\trans$ is a linear combination of the last two columns of $\Psi$, hence $\bx \in \Col(\Psi)$. Then $\supp(\bx) = \beta$.  Since the rows indexed by $\alpha$ form a linearly independent set, $A$ has the SSP with respect to $G + \beta$.  By Lemma~\ref{lemma:liberation}, there is a matrix $A'\in\mptn(G + \beta)$ with the SSP and $\spec(A') = \spec(A)=\{0^{(3)},4^{(2)}\}$.  
\end{example}

As we can see in Example~\ref{ex:libvec}, the Matrix Liberation Lemma (Lemma~\ref{lemma:liberation}) relies on a  vector $\bx$ in its assumption while the conclusion only uses the support of $\bx$.  That is, the exact values of entries of $\bx$ do not play a role.  Since the support of $\bx$ is all we need in practice, we aim to better understand, how to determine all possible supports of vectors $\bx$ that meet the requirements in the Matrix Liberation Lemma  in the next section.  In particular, Example~\ref{ex:k4k1} will explain how the vector $\bx$ in the example above can be found.

\section{Liberation set of a matrix}
\label{sec:libermtx}

To motivate the work in this section, we start by stating the main definition and the main result.

\begin{definition}
\label{def:liberation}
Let $G$ be a graph and $A\in \mptn(G)$. A nonempty set of edges  $\beta \subseteq E(\overline{G})$ is called an \emph{SSP liberation set of $A$} (or an \emph{SAP liberation set of $A$}, respectively) if and only if $A$ has the SSP (the SAP, respectively) with respect to $G+\beta'$ for all $\beta' \subset \beta$ with $|\beta'|=|\beta|-1$.
\end{definition}

\begin{lemma}[Matrix Liberation Lemma --- set version]
\label{lem:liberation}
Let $A$ be a matrix in $\mptn(G)$ and $\beta$ an SSP (an SAP, respectively) liberation set of $A$.  Then there is a matrix $A'\in\mptn(G + \beta)$ with the SSP (the SAP, respectively) such that $\spec(A) = \spec(A')$ (or $\rank(A) = \rank(A')$, respectively).
\end{lemma}

\begin{example}
\label{ex:triviallibset}
Let $G$ be a graph and $A = \begin{pmatrix} a_{i,j} \end{pmatrix}\in\mptn(G)$.  If $A$ has the SSP (with respect to $G$), then any nonempty subset $\beta\subseteq E(\overline{G})$ is an SSP liberation set of $A$ since $A$ has the SSP with respect to any spanning supergraph of $G$.  On the other extreme,  $\beta = E(\overline{G})$ is an SSP liberation set of $A$ unless there exist $\{i,j\}\in E(\overline{G})$ such that $[A,X^{i,j}] = O$.  This happens if and only if both of $i$ and $j$ are isolated vertices in $G$, and $a_{i,i} = a_{j,j}$.
\end{example}

In the remark below we give an outline of a proof of Lemma \ref{lem:liberation} that depends on Theorem~\ref{thm:supergraphH} and highlights how the SSP with respect to a graph is used in the definition of the liberation set. Later we will develop a more detailed proof that relies on  Lemma \ref{lemma:liberation} and makes the equivalence between the two results transparent. 

\begin{remark}
One may understand the liberation set as a set where every entry can be perturbed into a nonzero entry individually.  Suppose $A\in\mptn(G)$ has a liberation set $\beta$.  Let $\beta = \{e_1,\ldots, e_k\}$ and $\beta_i := \beta\setminus\{e_i\}$, $i=1,\ldots,k$.  Then by definition $A$ has the SSP with respect to $G + \beta_i$ for all $i \in [k]$.  According to the extended version of the Supergraph Lemma (Theorem~\ref{thm:supergraphH}), one may perturb $A$ to obtain $A_1$ that satisfies: 
\begin{itemize}
    \item 
    the entries in $A_1$ corresponding to $e_1$ become nonzero,
    \item 
    every nonzero entry in $A$ stays nonzero also in $A_1$,
    \item 
    off-diagonal entries of $A_1$ outside  $E(G) \cup \beta$ remain zero,
    \item 
    $A_1$ has the same spectrum as $A$,
    \item the new matrix still has the SSP with respect to $G + \beta_i$ for each $i$ where $A_1$ is zero on the entry of $e_i$.  
\end{itemize}
If $A_1$ is nonzero on the entry of $e_2$, then set $A_2 = A_1$.  
Otherwise, since $A_1$ has the SSP with respect to $G + \beta_2$, we may perturb $A_1$ into $A_2$ in the same way; in particular, the entry of $e_2$ becomes nonzero.  We may continue this process inductively to obtain $A_3, \ldots, A_k$.  Note that once an entry turns nonzero, it stays nonzero in this process.  Therefore, in the end we obtain a matrix $A' = A_k$ as desired in the Matrix Liberation Lemma (Lemma~\ref{lem:liberation}).
\end{remark}

We will show in Proposition~\ref{prop:betaliberationeq} that Lemma~\ref{lem:liberation} is in the correct setting equivalent to Lemma~\ref{lemma:liberation}. As we try to understand all possible supports of vectors in the column space of a given matrix, we depend on basic linear algebra methods recalled in the next lemma.

\begin{lemma}
\label{lem:mtxequiv}
For $M \in \mat[m,n]$ and $\alpha\subsetneq [m]$, the following conditions are equivalent.
\begin{enumerate}[label={\rm(\arabic*)}]
\item\label{alpha1} $M[\alpha\cup\{k\},:]$ has full row-rank for all $k\in [m]\setminus\alpha$. 
\item\label{alpha2} There exists a nonzero vector $\bv\in \Col(M)$ such that  $\supp(\bv) = [m]\setminus\alpha$ and $M[\alpha,:]$  has full row-rank.
\item\label{alpha3} The matrix obtained from $M$ by permuting the the rows labeled by $\alpha$ to the top has the column reduced echelon form
\[\begin{pmatrix}
 I_{|\alpha|} & O \\
 ? & B
\end{pmatrix},\]
where each row of $B$ is a nonzero vector.
\end{enumerate}
\end{lemma}
\begin{proof}
To prove that \ref{alpha1} implies \ref{alpha2}, suppose that for some set $\alpha \subsetneq [m]$, the matrix $M[\alpha\cup\{k\},:]$ has full row-rank for all $k\in [m]\setminus\alpha$.  This implies $\dim(\Col(M[\alpha\cup\{k\},:]))=|\alpha|+1$, and hence $\Col(M[\alpha\cup\{k\},:])=\mathbb{R}^{|\alpha|+1}. $ In particular, $\Col(M[\alpha\cup\{k\},:])$ contains the vector that vanishes on $\alpha$ and has the entry corresponding to $k$ equal to $1$.  Therefore, 
for each $k\in [m]\setminus\alpha$, we can find a vector $\bv_k \in \Col(M)$ satisfying $\bv_k[\alpha] = \bzero$ and $k$-th entry of $\bv_k$ is nonzero.   Hence, there exists a linear combination of such vectors
 \[\bv = \sum_{k\in [m]\setminus\alpha}c_k\bv_k \in \Col(M),\]
so that $\supp(\bv)=[m]\setminus \alpha$, which proves our claim.

Conversely, we suppose $\bv \in \Col(M)$ is a nonzero vector such that $[m]\setminus\supp(\bv) = \alpha$  and the matrix $M[\alpha,:]$ has full row-rank.  Let $k\in [m]\setminus\alpha = \supp(\bv)$ and  $\alpha' = \alpha\cup\{k\}$.  We will show $M' = M[\alpha',:]$ has full row-rank.  To see this, let $\bx\in\mathbb{R}^m$ be a vector that vanishes outside $\alpha'$ and $\bx\trans M = \bzero$.  Since $\bx\trans M = \bzero$ and $\bv\in\Col(M)$, we know $\bx\trans\bv = 0$.  Moreover, $\supp(\bx)\cap\supp(\bv)\subseteq \{k\}$ implies $\bx$ is zero at the $k$-th entry.  Therefore, $\bx$ vanishes outside $\alpha$, which leads to  $\bx = \bzero$ because $M[\alpha,:]$ has full row-rank.  In summary, $\bx = \bzero$ is the only vector in $\mathbb{R}^m$ that vanishes outside $\alpha'$, and satisfies $\bx\trans M = \bzero$, so $M'$ has full row-rank.  

  Next we will show that \ref{alpha1} $\implies$ \ref{alpha3} $\implies$ \ref{alpha2}.  Suppose $\alpha\subsetneq [m]$ is a set such that $M[\alpha\cup\{k\},:]$ has full row-rank for all $k\in [m]\setminus \alpha$.  Let $W$ be the matrix obtained from $M$ by permuting all rows in $\alpha$ to the top.  Recall that column operations do not change the linear dependency of rows.  Since $M[\alpha,:]$ has full row-rank, the first $|\alpha|$ rows of $W$ form a linearly independent set, and its column reduced echelon form has the form 
\[\begin{pmatrix}
 I_{|\alpha|} & O \\
 ? & B
\end{pmatrix}.\]
Moreover, since for $k\in [m]\setminus\alpha$, the rows of $W$ corresponding to $\alpha\cup\{k\}$ are also independent, and the row of $B$ corresponding to $k$ is a nonzero vector.  This establishes (3).  With this property, one may find a nowhere zero vector in the column space of $B$, which means there is a vector in the column space of $W$ whose support is $[m]\setminus\alpha$.  Therefore, there is a nonzero vector $\bv\in\Col(M)$ with $\supp(\bv) = [m]\setminus\alpha$.  Again, by the column reduced echelon form of $W$, the rows of $W$ in $\alpha$ are independent, so are the rows of $M$ in $\alpha$.
\end{proof}

In our application of Lemma \ref{lem:mtxequiv} we will take $M$ to be the verification matrix $\Psi$. Together with Lemma \ref{lemma:liberation} this will allow us to prove Lemma~\ref{lem:liberation}. 
Statement \ref{alpha3} in  Lemma~\ref{lem:mtxequiv}  is not required for the proof of the theorem, but it does provide an algorithmic way to determine the liberation sets.

\begin{proposition}\label{prop:betaliberationeq}
Let $G$ be a graph, $A \in \mptn(G)$, and $\Psi$ the SSP (the SAP, respectively) verification matrix for $A$. For a nonempty set $\beta \subseteq E(\overline{G})$ the following statements are equivalent:
\begin{enumerate}[label={\rm(\arabic*)}]
\item\label{beta1} $\beta$ is an SSP (an SAP, respectively) liberation set of $A$.
\item\label{beta2} $\Psi[\alpha\cup\{e\},:]$ has full row-rank for all $e\in \beta$, where $\alpha = E(\overline{G})\setminus\beta$.
\item\label{beta3} There exists a nonzero vector $\bx \in \Col(\Psi)$ such that $\supp(\bx)=\beta$ and $A$ has the SSP (the SAP, respectively) with respect to $G+\beta$. 
\item\label{beta4} By permuting the rows of $\beta$ in $\Psi$ to the bottom, its column reduced echelon form has the form 
\[\begin{pmatrix}
 I_{|\alpha|} & O \\
 ? & B
\end{pmatrix}\]
such that each row of $B$ is a nonzero vector, where $\alpha = E(\overline{G})\setminus \beta$.  
\end{enumerate}
\end{proposition}
\begin{proof}
By definition and Remark~\ref{rem:vermtx}, $\beta$ is a liberation set of $A$ precisely when $\Psi[E(\overline{G})\setminus\beta',:]$ has full row-rank for all  $\beta' \subset \beta$ with $|\beta'|=|\beta|-1$.  By writing $E(\overline{G})\setminus\beta$ as $\alpha$, the set $E(\overline{G})\setminus\beta'$ can be written as $\alpha\cup\{e\}$ for $e\in\beta$, we deduce that \ref{beta1} and \ref{beta2} are equivalent.  With Lemma~\ref{lem:mtxequiv} and Remark~\ref{rem:vermtx}, the statements \ref{beta2}, \ref{beta3}, and \ref{beta4} are equivalent.
\end{proof}

\begin{proof}[Proof of Lemma~\ref{lem:liberation}]
\label{proof:liberation-lemma} According to the equivalence of \ref{beta1} and \ref{beta3} in Proposition~\ref{prop:betaliberationeq}, the claims in Lemma~\ref{lem:liberation} are equivalent to the ones in Lemma~\ref{lemma:liberation}.
\end{proof}

\begin{example}
\label{ex:k4k1}
Let $G$, $A$, $\Psi$, $\alpha$, and $\beta$ be as in Example~\ref{ex:libvec}, and recall that in the labeling of the rows of $\Psi$ the first two rows are labeled by the elements of $\alpha$. Let us have a look at all the equivalent conditions in Proposition~\ref{prop:betaliberationeq} for this example. 
The column reduced echelon form of $\Psi$ is 
\[\Psi' = \left(\begin{array}{cc|cccccccc}
1 & 0 & 0 & 0 & 0 & 0 & 0 & 0 & 0 & 0 \\
0 & 1 & 0 & 0 & 0 & 0 & 0 & 0 & 0 & 0 \\
\hline
0 & 0 & 1 & 0 & 0 & 0 & 0 & 0 & 0 & 0 \\
-1 & -1 & -1 & 0 & 0 & 0 & 0 & 0 & 0 & 0
\end{array}\right) = \begin{pmatrix}
 I_2 & O \\
 ? & B
\end{pmatrix}.\]
Since $B$ has no zero rows $\Psi$ satisfies the condition (4) in Proposition~\ref{prop:betaliberationeq}, for our chosen $\alpha$ and $\beta$. Condition (3) holds for example for  $\bx = (0,0,1,-1)\trans \in \Col(\Psi) = \Col(\Psi')$, by noting that $\supp(\bx) = \beta$ and $A$ has the SSP with respect to $G + \beta$. Since column operations do not change the dependency of rows, we can easily deduce from $\Psi'$ that both $\Psi[\alpha \cup \{\{3,5\}\},:]$ and $\Psi[\alpha \cup \{\{4,5\}\},:]$ have full row-rank, hence conditions (1) and (2) hold.

As a consequence of all these equivalent conditions, there is a matrix $A'\in\mptn(G + \beta)$ with $\spec(A') = \{0^{(3)}, 4^{(2)}\}$.  
\end{example}

The matrix $A$ in the example above is small enough that we were able to write out the SSP verification matrix, and investigate all the equivalent properties of Proposition \ref{prop:betaliberationeq}. However, the final conclusion on the existence of the matrix $A'$ is not very exciting. Our next example is only slightly more involved, but it already resolves an open question in the IEP-$G$.

\begin{example}\label{ex:g151}
Consider a family of matrices of the form:
$$A=\left(
\begin{array}{cccccc}
 b & t & t & t & 0 & 0 \\
 t & 0 & 0 & 0 & 0 & 0 \\
 t & 0 & 0 & 0 & 0 & 0 \\
 t & 0 & 0 & 0 & 0 & 0 \\
 0 & 0 & 0 & 0 & a & a\\
 0 & 0 & 0 & 0 & a & a \\
\end{array}
\right) \in \mptn(K_{1,3}\dunion K_{2}),$$
where $a,b,t \in \mathbb{R}\setminus\{0\}$. With the observation that the matrix $A$ has eigenvalues equal to $\{0^{(3)}, 2a, \frac{1}{2} (b - \sqrt{b^2 + 12 t^2}), \frac{1}{2} (b + \sqrt{b^2 + 12 t^2})\}$, it is straightforward to check that this family of matrices together with their translations realizes the ordered multiplicity lists $(1,3,2)$, $(2,3,1)$, $(1,1,3,1)$ and $(1,3,1,1)$ spectrally arbitrary.  
We claim that $\beta=\{\{3,5\},\{4,5\},\{2,6\},\{4,6\}\}$ is a liberation set of all matrices in this family. Using symmetries of $(K_{1,3}\dunion K_{2}) + \beta$ it is enough to check that $A$ has the SSP with respect to  $\beta_1=\{\{3,5\},\{4,5\},\{4,6\}\}$ and $\beta_2=\{\{3,5\},\{4,5\},\{2,6\}\}$, which is a straightforward (albeit tedious) calculation exercise. Also note that $\beta$ is not a liberation set of every matrix in  $\mptn(K_{1,3}\dunion K_{2})$.
We conclude that the multiplicity lists listed above can be realized spectrally arbitrarily for $\mathsf{G_{151}}=(K_{1,3}\dunion K_{2})+\beta$, shown on Figure~\ref{fig:g151}. 

\begin{figure}[h]
\centering
\begin{tikzpicture}
\foreach \i in {1,2} {
    \pgfmathsetmacro{\angle}{60 * (\i)}
    \node[label={\angle:$\i$}] (\i) at (\angle:1) {};
 }
\node[label={0:3}] (3) at (0:1) {};
\node[label={-60:4}] (4) at (-60:1) {};
\node[label={-120:5}] (5) at (-120:1) {};
\node[label={180:6}] (6) at (180:1) {};
\draw (3)--(1) -- (2);
\draw (1) -- (4);
\draw (5) -- (6);
\draw[color=blue,thick,dashed] (3) -- (5) -- (4);
\draw[color=blue,thick,dashed] (2) -- (6) -- (4);
\node[rectangle,draw=none] at (0,-1.5) {$\mathsf{G_{151}}$};
\end{tikzpicture}
\caption{Graph $\mathsf{G_{151}}=(K_{1,3}\dunion K_{2})+\beta$ on six vertices with $\beta = \{\{3,5\},\{4,5\},\{2,6\},\{4,6\}\}$.}
\label{fig:g151}
\end{figure}
\end{example}

\section{Direct sum of matrices with the SSP}
\label{sec:directsum}

As an application of Lemma~\ref{lem:liberation}, we consider matrices of the form $A \oplus B$ with $A \in \mptn(G)$ and $B \in \mptn(H)$.

\begin{definition}
Let $G$ and $H$ be graphs on $m$ and $n$ vertices, respectively, $\beta$ a set of edges between $G$ and $H$, and $Y$ an $m \times n$ matrix. We say that \emph{ $Y$ vanishes on $\beta$}, if under the labeling of the rows of $Y$ by $V(G)$ and the columns  of $Y$ by $V(H)$ we have $Y_{u,v}=0$ for all $\{u,v\} \in \beta$. We denote this by $Y\vert_{\beta}=0$.
\end{definition}

\begin{proposition}\label{prop:directsxp}
Let $G$ and $H$ be graphs on $m$ and $n$ vertices, respectively, $A\in\mptn(G)$, $B\in\mptn(H)$, and $\beta$ a set of edges between $G$ and $H$.
 \begin{enumerate}
     \item Assume that $A$ and $B$ have the SSP, and that  $Y = O$ is the only matrix in $ \mathbb{R}^{m \times n}$ with $Y\vert_{\beta}=0$ that satisfies $AY - YB = O$. Then $A\oplus B\in\mptn(G\dunion H)$ has the SSP with respect to $(G\dunion H) + \beta$.
      \item Assume that $A$ and $B$ have the SAP, and that  $Y = O$ is the only matrix in $ \mathbb{R}^{m \times n}$ with $Y\vert_{\beta}=0$, that satisfies $AY = YB = O$. Then $A\oplus B\in\mptn(G\dunion H)$ has the SAP with respect to $(G\dunion H) + \beta$.
     \end{enumerate}
\end{proposition}

\begin{proof}
Let 
\[X = \begin{pmatrix}
 X_A & Y \\
 Y\trans & X_B
\end{pmatrix}\]
be a symmetric matrix such that $X \in \mptnclo(\overline{G\dunion H + \beta})$.
To prove the first item we assume that  $[A\oplus B, X] = O$. From the SSP condition on $A$ and $B$ this reduces to $AY-YB=O$, and the conclusion follows. 
Similarly, to prove the second item, we assume $(A\oplus B)X = O$. The SAP assumption for $A$ and $B$ reduces this equality to $AY=YB= O$, which implies $Y=O$ and the second statement is proved. 
\end{proof}

The following standard linear algebra result will help us to understand the set of solutions $Y$ to the equation $AY - YB = O$. The proof is added for completeness.
 For any two  matrices $A\in\mat[m]$ and $B\in\mat[n]$ we denote by $\mathcal{R}(A,B)$   the set of solutions  $Y\in \mat[m,n]$ of the equation $AY - YB = O$. The following proposition describes $\mathcal{R}(A,B)$.
 
\begin{proposition}
\label{prop:twist}
Let $A$ and $B$ be symmetric matrices of order $m$ and $n$, respectively.  Suppose $A$ and $B$ have $k$ distinct common eigenvalues $\lambda_1,\ldots,\lambda_k$ with $\mult_A(\lambda_i) = a_i$ and $\mult_B(\lambda_i) = b_i$ for $i = 1,\ldots, k$.  Then 
\[\mathcal{R}(A,B) = \vspan\{\bu\bv\trans: A\bu = \lambda_i\bu,\ B\bv = \lambda_i\bv \text{ for some }i \in [k]\},\]
and has the dimension $\sum_{i\in [k]} a_ib_i$.  
\end{proposition}

\begin{proof}
Let $Q_A$ and $Q_B$ be orthogonal matrices that diagonalize $A$ and $B$: $Q_A\trans AQ_A = D_A$ and $Q_B\trans BQ_B = D_B$.  Then the equation $AY - YB = O$ is equivalent to $D_AY' - Y'D_B = O$, where $Y' = Q_A\trans YQ_B$.  We deduce that the $(i,j)$-entry of $Y'$ has to be zero if the $(i,i)$-entry $(D_A)_{i,i}$ of $D_A$ is different from the $(j,j)$-entry $(D_B)_{j,j}$ of $D_B$.  Let $\be_i$ and $\bff_i$ denote the $i$-th columns of $I_m$ and $I_n$, respectively.  With this notation we have  
\[\mathcal{R}(D_A,D_B)=\vspan\{\be_i\bff_j\trans : (D_A)_{i,i} = (D_B)_{j,j}\}.\]
Since $\mathcal{R}(A,B)=Q_A\mathcal{R}(D_A,D_B)Q_B\trans$, the statement follows.
\end{proof}

 If both $A$ and $B$ have the SSP and they do not have any common eigenvalues, then $A\oplus B$ has the SSP, see Theorem~\ref{thm:supergraph}. However, if $A$ and $B$ share one distinct eigenvalue, this no longer holds, and we aim to understand liberation sets for $A \oplus B$ in this case. To this end we will impose some conditions on the eigenspaces of $A$ and $B$.
 
\begin{definition}
\label{def:generic}
Let $\mathcal{W}$ be a $d$-dimensional subspace of $\mathbb{R}^n$ and $W$ be an $n\times d$ matrix whose columns form a basis of $\mathcal{W}$. Then $\mathcal{W}$ is said to be \emph{generic} if and only if every $d\times d$ submatrix of $W$ is invertible.
\end{definition}

The generic subspace is well-defined above, as it does not depend on the choice of $W$. Indeed, if $W_1$ and $W_2$ are $n \times d$ matrices and columns of $W_1$ and columns of $W_2$ form bases of $\mathcal{W}$, then $W_2 = W_1Q$ for some invertible  $Q\in \mat[d]$.  Thus, a $d\times d$ submatrix $B$ in $W_1$ is invertible if and only if the corresponding submatrix $BQ$ in $W_2$ is invertible.  As an example, when $\mathcal{W}$ is a $1$-dimensional subspace of $\mathbb{R}^n$, then it is generic if and only if it is spanned by a nowhere zero vector.

\begin{theorem}
\label{thm:multi-multi-eigenvalue-in-common}
Let $A\in \mptn(G)$ and $B \in \mptn(H)$ have the SSP.  Suppose $\spec(A)\cap \spec(B)=\{\lambda\}$ with $\mult_A(\lambda)=k$ and $\mult_B(\lambda)=\ell$, where $\ker(A - \lambda I)$ and $\ker(B - \lambda I)$ are generic. 
Then for any $V_G \subseteq V(G)$ and $V_H\subseteq V(H)$ with
\begin{itemize}
\item $|V_G|=k$ and $|V_H|=\ell+1$ or 
\item $|V_G|=k+1$ and $|V_H|=\ell$,
\end{itemize}
the set 
\[\beta = \{\{u,v\}: u\in V_G \text{ and }v\in V_H\}\]
is an SSP liberation set of $A\oplus B$.  In particular, there exists a matrix $C\in S(G\dunion H + \beta)$ with the SSP such that  $\spec(C) = \spec(A)\cup \spec(B)$.
\end{theorem}

\begin{proof}
Let $U_A$ and $U_B$ be $|V(G)|\times k$ and $|V(H)|\times \ell$  matrices whose columns form bases of $\ker(A - \lambda I)$ and $\ker(B - \lambda I)$, respectively.  Let $\beta \subset V(G) \times V(H)$ be as in the statement of the theorem, and $\beta'\subset\beta$ with $|\beta'| = |\beta| - 1$. Observe that by the assumption on $V_G$ and $V_H$, $\beta'$ contains a $k\times \ell$ grid, say $\mathcal{I}\times\mathcal{J}$, regardless the choice of $\beta'$. 

By Proposition~\ref{prop:twist}, 
\[\mathcal{R}(A,B) = \{U_ASU_B\trans \colon S\in\mat[k,\ell]\}.\]
Suppose $Y 
\in \mathcal{R}(A,B)$ satisfies $Y\vert_{\mathcal{I} \times \mathcal{J}}=0$. Then $Y = U_AS_YU_B$ for some $S_Y\in\mat[k,\ell]$, and  
\[U_A[\mathcal{I},:] S_Y U_B[\mathcal{J},:]\trans = O.\]
Since both $\ker(A - \lambda I)$ and $\ker(B - \lambda I)$ are assumed to be generic, $U_A[\mathcal{I},:]$ and $U_B[\mathcal{J},:]$ are invertible, so $S_Y = O$. This implies $Y = O$ and proves that $A\oplus B$ has the SSP with respect to $G\dunion H + \beta'$ by Proposition~\ref{prop:directsxp}. This also proves that $\beta$ is an SSP liberation set of $A\oplus B$,  since our choice of $\beta'$ was arbitrary. The conclusion follows by Lemma~\ref{lem:liberation}. 
\end{proof}

When applying Theorem \ref{thm:multi-multi-eigenvalue-in-common}, the assumption on genericity of eigenspaces can be hard to prove. However, this  condition is straightforward to check for specific matrices of small dimension. 

\begin{example}
Let $A\in\mptn(C_6)$ and $B\in\mptn(C_8)$ be 
\[
\begin{pmatrix}
 0 & 1 & 0 & 0 & 0 & -1 \\
 1 & 0 & 1 & 0 & 0 & 0 \\
 0 & 1 & 0 & 1 & 0 & 0 \\
 0 & 0 & 1 & 0 & 1 & 0 \\
 0 & 0 & 0 & 1 & 0 & 1 \\
 -1 & 0 & 0 & 0 & 1 & 0 \\
\end{pmatrix} 
\text{ and }
\begin{pmatrix}
0 & 1 & 0 & 0 & 0 & 0 & 0 & -\sqrt{\frac{5}{3}} \\
 1 & 0 & 1 & 0 & 0 & 0 & 0 & 0 \\
 0 & 1 & 0 & 1 & 0 & 0 & 0 & 0 \\
 0 & 0 & 1 & 0 & \sqrt{\frac{5}{3}} & 0 & 0 & 0 \\
 0 & 0 & 0 & \sqrt{\frac{5}{3}} & 0 & 1 & 0 & 0 \\
 0 & 0 & 0 & 0 & 1 & 0 & 1 & 0 \\
 0 & 0 & 0 & 0 & 0 & 1 & 0 & 1 \\
 -\sqrt{\frac{5}{3}} & 0 & 0 & 0 & 0 & 0 & 1 & 0 \\
\end{pmatrix},
\]
respectively.  
Then 
\[\begin{aligned}
 \spec(A) &= \left\{-\sqrt{3}^{(2)}, 0^{(2)}, \sqrt{3}^{(2)}\right\} \text{ and} \\
 \spec(B) &= \left\{-2^{(2)}, -\sqrt{\frac{2}{3}}^{(2)}, \sqrt{\frac{2}{3}}^{(2)}, 2^{(2)}\right\}.
\end{aligned}\]
Name the distinct eigenvalues of $A$ and $B$ as $\lambda_1 < \lambda_2 < \lambda_3$ and $\mu_1 < \cdots < \mu_4$, respectively.  By direct computation, every eigenspace of $A$ and $B$ is generic except for $\ker(A)$, the eigenspace of $A$ with respect to $\lambda_2 = 0$.  The eigenspaces of $A$ and $A + sI$ are the same for any $s$. Hence, by choosing $s = -\lambda_1 + \mu_1$, the matrices $A + sI$ and $B$ have a unique common eigenvalues $\mu_1$ and meet the requirement in Theorem~\ref{thm:multi-multi-eigenvalue-in-common}.  Thus, we may choose a graph $H$ obtained from $C_6 \dunion C_8$ by joining two vertices in $V(C_6)$ to three vertices in $V(C_8)$ or three vertices in $V(C_6)$ to two vertices in $V(C_8)$. (See some of examples of such graphs in Figure~\ref{fig:C6C8}.)  Then Theorem~\ref{thm:multi-multi-eigenvalue-in-common} guarantees a matrix $M\in\mptn(H)$ with the SSP and the ordered multiplicity list $(4,2,2,2,2,2)$.  Indeed, by choosing $s$ as $-\lambda_1 + \mu_2$, $-\lambda_1 + \mu_3$, $-\lambda_3 + \mu_1$, $-\lambda_3 + \mu_1$, or $-\lambda_3 + \mu_1$, every ordered multiplicity list with the unordered multiplicity list $\{4,2,2,2,2,2\}$ is realizable by a matrix in $\mptn(H)$ with the SSP.

\begin{figure}[h]
\centering
\begin{tikzpicture}
\foreach \i in {1,...,6} {
    \pgfmathsetmacro{\angle}{60 * (\i - 1)}
    \node[fill=white] (\i) at (\angle:0.5) {};
}
\foreach \i in {7,...,14} {
    \pgfmathsetmacro{\angle}{45 * (\i-7)}
    \node[fill=white] (\i) at (\angle:1) {};
 }
\draw (1) -- (2) -- (3) -- (4) -- (5) -- (6) -- (1);
\draw (7) -- (8) -- (9) -- (10) -- (11) -- (12) -- (13) -- (14) -- (7); 
\foreach \i in {2,3,8,9,10} {
    \node [fill=blue,thick,dashed] at (\i) {};
}
\foreach \i in {2,3} {
    \foreach \j in {8,9,10} {
        \draw[blue, thick,dashed] (\i) -- (\j);
    }
}
\end{tikzpicture}
\hfil
\begin{tikzpicture}
\foreach \i in {1,...,6} {
    \pgfmathsetmacro{\angle}{60 * (\i - 1)}
    \node[fill=white] (\i) at (\angle:0.5) {};
}
\foreach \i in {7,...,14} {
    \pgfmathsetmacro{\angle}{45 * (\i-7)}
    \node[fill=white] (6+\i) at (\angle:1) {};
}
\draw (1) -- (2) -- (3) -- (4) -- (5) -- (6) -- (1);
\draw (7) -- (8) -- (9) -- (10) -- (11) -- (12) -- (13) -- (14) -- (7); 
\foreach \i in {2,3,8,10,13} {
    \node [fill=blue,thick,dashed] at (\i) {};
}
\foreach \i in {2,3} {
    \foreach \j in {8,10,13} {
        \draw[blue, thick,dashed] (\i) -- (\j);
    }
}
\end{tikzpicture}
\hfil
\begin{tikzpicture}
\foreach \i in {1,...,6} {
    \pgfmathsetmacro{\angle}{60 * (\i - 1)}
    \node[fill=white] (\i) at (\angle:0.5) {};
}
\foreach \i in {7,...,14} {
    \pgfmathsetmacro{\angle}{45 * (\i-7)}
    \node[fill=white] (\i) at (\angle:1) {};
}
\draw (1) -- (2) -- (3) -- (4) -- (5) -- (6) -- (1);
\draw (7) -- (8) -- (9) -- (10) -- (11) -- (12) -- (13) -- (14) -- (7); 
\foreach \i in {2,3,6,9,13} {
    \node [fill=blue,thick,dashed] at (\i) {};
}
\foreach \i in {2,3,6} {
    \foreach \j in {9,13} {
        \draw[blue, thick,dashed] (\i) -- (\j);
    }
}
\end{tikzpicture}
\hfil
\begin{tikzpicture}
\foreach \i in {1,...,6} {
    \pgfmathsetmacro{\angle}{60 * (\i - 1)}
    \node[fill=white] (\i) at (\angle:0.5) {};
}
\foreach \i in {7,...,14} {
    \pgfmathsetmacro{\angle}{45 * (\i-7)}
    \node[fill=white] (\i) at (\angle:1) {};
}
\draw (1) -- (2) -- (3) -- (4) -- (5) -- (6) -- (1);
\draw (7) -- (8) -- (9) -- (10) -- (11) -- (12) -- (13) -- (14) -- (7); 
\foreach \i in {2,3,6,8,9} {
    \node [fill=blue,thick,dashed] at (\i) {};
}
\foreach \i in {2,3,6} {
    \foreach \j in {8,9} {
        \draw[blue, thick,dashed] (\i) -- (\j);
    }
}

\end{tikzpicture}
\caption{All the graphs in this figure have $14$ vertices, $20$ edges, and every ordered multiplicity list corresponding the unordered multiplicity list $\{4,2,2,2,2,2\}$ is realizable by a matrix with the SSP.
}
\label{fig:C6C8}
\end{figure}
\end{example}

All the liberation sets identified in Theorem~\ref{thm:multi-multi-eigenvalue-in-common} are rectangular grids.  In the case of $\ell = 1$, we can do better, as we show in the next theorem.

\begin{theorem}\label{thm:multi-1-eigenvalue-in-common}
Let $A\in \mptn(G)$ and $B \in \mptn(H)$ have the SSP.  Suppose $\spec(A)\cap \spec(B)=\{\lambda\}$ with $\mult_A(\lambda)=k$ and $\mult_B(\lambda)=1$, where $\ker(A - \lambda I)$ and $\ker(B - \lambda I)$ are generic.  If either 
\begin{itemize}
\item $|\beta \cap (\{u\}\times V(H))| = 2$ for $k$ distinct $u\in V(G)$, or
\item $|\beta\cap (V(G)\times \{v\})| = k+1$ for some $v\in V(H)$,
\end{itemize}
then the set $\beta$ is an SSP liberation set of $A\oplus B$. In particular, there is a matrix $C\in S(G\dunion H + \beta)$ with the SSP such that  $\spec(C) = \spec(A)\cup \spec(B)$.

Moreover, when $k = 1$,  any $\beta\subseteq V(G)\times V(H)$ with $|\beta| = 2$ is an SSP liberation set of $A\oplus B$.  
\end{theorem}
\begin{proof}
The case when $|\beta\cap (V(G)\times \{v\})| = k+1$ for some $v\in V(H)$ follows from Theorem~\ref{thm:multi-multi-eigenvalue-in-common}, so we focus on the case when $|\beta \cap (\{u\}\times V(H))| = 2$ for $k$ distinct $u\in V(G)$.  Note that any $\beta'\subset\beta$ with $|\beta'| = |\beta| - 1$ satisfies $|\beta' \cap (\{u\}\times V(H))| = 1$ for $k$ distinct $u\in V(G)$. 

 Let $U_A$ be a $n\times k$ matrix whose columns form a basis of $\ker(A - \lambda I)$, and $\ker(B - \lambda I) = \vspan\{\bv\}$ for some vector $\bv$.  Now $Y \in \mathcal{R}(A,B)$ can be written as 
\[Y = U_A\bs_Y\bv\trans\]
for some $\bs_Y \in \mat[k,1]$.  If $Y\vert_{\beta'}=0$, then at least $k$ rows of $Y$ have a zero entry.  However, since $\bv$ is nowhere zero by the genericity of $\ker(B - \lambda I)$, each row of $Y$ is either the zero vector or a nowhere zero vector.  Therefore, $Y$ contains at least $k$ zero rows.  The genericity of $U_A$ now implies $\bs_Y = \bzero$ and $Y = O$.  Since this argument applies to any $\beta'\subset\beta$ with $|\beta'| = |\beta| - 1$, we have proved that $\beta$ is an SSP liberation set of $A\oplus B$.   

In the case of $k = 1$, we may further assume $\ker(A - \lambda I) = \vspan\{\bu\}$ for some nowhere zero vector $\bu$ and $Y = \bu s_Y\bv\trans$ for some scalar $s_Y$.  As long as $Y$ contains a zero entry, we know $s_Y = 0$ and $Y = O$.  Therefore, any $\beta\subseteq V(G)\times V(H)$ with $|\beta| = 2$ is an SSP liberation set of $A\oplus B$.  
\end{proof}

\begin{remark}
\label{rem:beinggeneric}
Below we list some situations in which we can prove genericity of eigenspaces:
\begin{itemize}
\item Let $T$ be a tree and $A\in\mptn(G)$.  Then the smallest and the largest eigenvalues of $A$ must be simple by the Perron--Frobenius Theorem, and their eigenspaces are generic.
\item Let $G$ be a connected graph on $n$ vertices and $\lambda_1 < \cdots < \lambda_n$.  Then there is a matrix $A\in\mptn(G)$ with $\spec(A) = \{\lambda_1,\ldots,\lambda_n\}$ such that the eigenspace of each eigenvalue is generic \cite{MR3506498}.  
\item Let $G$ be a complete graph $K_n$ and $A\in\mptn(G)$.  By choosing appropriate orthogonal matrices $Q$, we may replace $A$ by $Q\trans AQ$ and assume the eigenspace of each eigenvalue is generic.  
\end{itemize}
\end{remark}

\begin{example}
\label{ex:twotrees}
Let $T_1$ and $T_2$ be trees and $H$ a graph obtained from $T_1\dunion T_2$ by adding two arbitrary edges between $V(T_1)$ and $V(T_2)$.  Suppose $A\in\mptn(T_1)$ and $B\in\mptn(T_2)$ are matrices with the SSP and the ordered multiplicity lists $(1, m_2,\ldots, m_{q_1-1},1)$ and $(1,r_2,\ldots, r_{q_2-1},1)$, respectively.  Then both
\[(2,m_2,\ldots, m_{q_1-1},1, r_2, \ldots, r_{q_2-1}1) \text{ and } (1,m_{q_1-1}, \ldots, m_2, 2, r_2, \ldots, r_{q_2-1},1,)\] 
are ordered multiplicity lists that are realizable by a matrix in $\mptn(H)$ with the SSP by Theorem~\ref{thm:multi-1-eigenvalue-in-common}.

In particular, it is known that $K_{1,3}$ and $P_2$ are SSP graphs with each realizable ordered multiplicity being spectrally arbitrary; see~\cite{MR4074182}. So, for arbitrary choice of real numbers $\lambda_1<\lambda_2<\lambda_3<\lambda_4$, let  $A\in \mptn(K_{1,3})$ be a matrix with the SSP and $\spec(A)=\{\lambda_1,\lambda_2^{(2)},\lambda_3\}$,  and let $B\in \mptn(P_{2})$ be a matrix with the SSP and $\spec(B)=\{\lambda_3,\lambda_4\}$.  
 Then by Theorem~\ref{thm:multi-1-eigenvalue-in-common}, the ordered multiplicity list $(1,2,2,1)$ is spectrally arbitrary and realizable by a matrix with the SSP for the graph $\mathsf{G_{100}}$ shown on Figure~\ref{fig:g100}. This resolves one of the multiplicity lists in~\cite[Appendix~B]{MR4284782}.
 
\begin{figure}[h]
\centering
\begin{tikzpicture}
\foreach \i in {1,...,6} {
    \pgfmathsetmacro{\angle}{60 * (\i - 1)}
    \node[label={\angle:$\i$}] (\i) at (\angle:1) {};
 }
\draw (3)--(1) -- (2);
\draw (4) -- (1);
\draw (5)-- (6);
\draw[color=blue,thick,dashed] (5) --(4) -- (6);
\node[rectangle,draw=none] at (0,-1.5) {$\mathsf{G_{100}}$};
\end{tikzpicture}
\caption{Graph $\mathsf{G_{100}} = K_{1,3}\dunion P_2 + \{\{4,5\},\{4,6\}\}$ on six vertices.}
\label{fig:g100}
\end{figure}
\end{example}

\begin{example}\label{ex:g127g169}
Let $G$ and $H$ be graphs.  Suppose $A\in\mptn(G)$ and $B\in\mptn(H)$ both have the SSP with $\spec(A) = \{\theta\}\cup\sigma$ and $\spec(B) = \{\theta\}\cup\tau$, where $\theta\notin\sigma \cup \tau$ and $\sigma\cap\tau = \emptyset$.  Moreover, assume that the eigenspaces of $A$ and $B$ corresponding to $\theta$ are generic. 

By Remark~\ref{rem:beinggeneric}, such matrices $A$ and $B$ exist for any connected graphs $G$ and $H$, if $\sigma$ and $\tau$ are sets of distinct real numbers of appropriate sizes. Furthermore, if $G$ is a complete graph, then $A$ exists for multiset $\sigma$ of real numbers of appropriate size, and a similar statement holds for $H$ and $B$.   Let $\beta\subseteq V(G)\times V(H)$ with $|\beta| = 2$.  By Theorem~\ref{thm:multi-1-eigenvalue-in-common}, there is a matrix in $\mptn(G\dunion H + \beta)$ with the SSP and spectrum $\{\theta^{(2)}\}\cup\sigma\cup\tau$.  

In particular, this setup applies to graphs $\mathsf{G_{127}}$ and $\mathsf{G_{169}}$ in Figure~\ref{fig:g127g169} for the following choices of $G$, $H$ and $\beta$:

\begin{itemize}
\item Let $G=K_3$ with $V(G)=[3]$, $H=P_3$ with $V(H)=3+[3]$, and $\beta = \{\{1,6\}, \{3,4\}\}$.  Since $\mathsf{G_{127}} = K_3\dunion P_3 + \beta$, we now know that the unordered multiplicity list $\{2,2,1,1\}$ is spectrally arbitrary for $\mathsf{G_{127}}$.
\item Taking $G=K_4$, $V(G)=[4]$, $H=K_2$, $V(H)=4+[2]$, and $\beta = \{\{4,5\}, \{2,6\}\}$, we have $\mathsf{G_{169}} = K_4\dunion K_2 + \beta$ hence the unordered multiplicity list $\{3,2,1\}$ is spectrally arbitrary for $\mathsf{G_{169}}$.
\end{itemize}

\begin{figure}[h]
\centering
\begin{tikzpicture}
\foreach \i in {1,...,6} {
    \pgfmathsetmacro{\angle}{60 + 60 * (\i - 1)}
    \node[label={\angle:$\i$}] (\i) at (\angle:1) {};
 }

\draw (1) -- (2) -- (3) -- (1);
\draw (4) -- (5) -- (6);
\draw[color=blue,thick,dashed] (1) -- (6);
\draw[color=blue,thick,dashed] (3) -- (4);
\node[rectangle,draw=none] at (0,-1.5) {$\mathsf{G_{127}}$};
\end{tikzpicture}
\hfil
\begin{tikzpicture}
\foreach \i in {1,...,6} {
    \pgfmathsetmacro{\angle}{-60 + 60 * (\i - 1)}
    \node[label={\angle:$\i$}] (\i) at (\angle:1) {};
 }
\draw (1) -- (2) -- (3) -- (4) -- (1);
\draw (1) -- (3);
\draw (2) -- (4);
\draw (5) -- (6);
\draw[color=blue,thick,dashed] (4) -- (5);
\draw[color=blue,thick,dashed] (2) -- (6);
\node[rectangle,draw=none] at (0,-1.5) {$\mathsf{G_{169}}$};
\end{tikzpicture}

\caption{Graphs $\mathsf{G_{127}} = K_3\dunion P_3 + \{\{1,6\},\{3,4\}\}$ and $\mathsf{G_{169}} = K_4\dunion K_2 + \{\{2,6\},\{4,5\}\}$ on six vertices.}
\label{fig:g127g169}
\end{figure}
These results resolve the question of spectral arbitrariness for ordered multiplicity lists of $\mathsf{G_{127}}$ and $\mathsf{G_{169}}$ listed in \cite[Appendix~B]{MR4284782}.  
\end{example}  

\begin{figure}[h]
\centering
\begin{tikzpicture}
\foreach \i in {1,...,6} {
    \pgfmathsetmacro{\angle}{60 * (\i - 1)}
    \node[label={\angle:$\i$}] (\i) at (\angle:1) {};
 }

\draw (1) -- (2) -- (3) -- (1);
\draw (4) -- (5) -- (6);
\draw[color=blue,thick,dashed] (6) -- (1) -- (5) -- (3) -- (4);
\node[rectangle,draw=none] at (0,-1.5) {$\mathsf{G_{163}}$};
\end{tikzpicture}
\caption{Graph $\mathsf{G_{163}} = K_3\dunion P_3 + \beta$ on six vertices with the set $\beta = \{\{1,6\}, \{1,5\}, \{3,5\}, \{3,4\}\}$.}
\label{fig:g163}
\end{figure}

\begin{example}
\label{ex:g163}
Let $G = K_3$, $V(G)=[3]$, $H = P_3$, $V(H)=3+[3]$, 
\[\beta = \{\{1,6\}, \{1,5\}, \{3,5\}, \{3,4\}\},\]
and $\theta, \lambda_1, \lambda_2, \mu_1\in\mathbb{R}$ distinct real numbers. 
By Remark~\ref{rem:beinggeneric}, there exist matrices $A\in\mptn(G)$ and $B\in\mptn(H)$ with the SSP, with $\spec(A) = \{\theta^{(2)}, \mu_1\}$ and $\spec(B) = \{\theta, \lambda_1,\lambda_2\}$ such that the eigenspaces of $A$ and $B$ with respect to $\theta$ are both generic.  

Then the graph $\mathsf{G_{163}}$ in Figure~\ref{fig:g163} is isomorphic to $G\dunion H + \beta$.  
Notice that $\beta$ contains two elements of the form $\{1,\cdot\}$ and two elements of the form $\{3,\cdot\}$, so there exists a matrix in $\mptn(\mathsf{G}_{163})$ with the SSP and the spectrum $\{\theta^{(3)}, \lambda_1, \lambda_2, \mu_1\}$ by Theorem~\ref{thm:multi-1-eigenvalue-in-common}.  Since $\theta$, $\lambda_1$, $\lambda_2$ and $\mu_1$ are arbitrary distinct numbers, the unordered multiplicity list $\{3,1,1,1\}$ is spectrally arbitrary, which resolves the question of spectral arbitrariness for the some of previously unresolved ordered multiplicity lists of $\mathsf{G_{163}}$ in \cite[Appendix~B]{MR4284782}.
\end{example}

We end this section with examples where the eigenspaces are not generic.

\begin{figure}[h]
\centering
\begin{tikzpicture}
\node (1) at (-2,0) {};
\node (2) at (-1,1) {};
\node (3) at (-1,0) {};
\node (4) at (-1,-1) {};
\node (5) at (2,0) {};
\node (6) at (1,1) {};
\node (7) at (1,0) {};
\node (8) at (1,-1) {};
\draw (1) -- (2); 
\draw (1) -- (3); 
\draw (1) -- (4);
\draw (5) -- (6); 
\draw (5) -- (7); 
\draw (5) -- (8);
\foreach \i in {2,3} {
    \foreach \j in {6,7,8} {
        \draw[blue,thick,dashed] (\i) -- (\j);
    }
}
\end{tikzpicture}
\caption{An example where the liberation set does not rely on generic eigenspaces.}
\label{fig:k13k13}
\end{figure}

\begin{example}
Consider $K_{1,3}$ with $V(K_{1,3})=[4]$, where $1$ is the vertex of degree $3$.  Let $A,B \in \mptn(K_{1,3})$ such that $\spec(A) = \{\lambda_1, \theta^{(2)}, \lambda_2\}$ and $\spec(B) = \{\mu_1, \theta^{(2)}, \mu_2\}$ with $\lambda_1 < \theta < \lambda_2$ and $\mu_1 < \theta < \mu_2$.  
By, e.g., \cite{MR4074182}, such spectra are realizable by matrices with the SSP for any distinct real numbers $\lambda_1$, $\lambda_2$, $\mu_1$, $\mu_2$, and $\theta$.
 
Then by interlacing and \cite[Lemma~5.1]{MR3567513} we have
\[A = \begin{pmatrix} 
 a & \ba\trans \\ 
 \ba & \theta I_3
\end{pmatrix} \text{ and }
B = \begin{pmatrix} 
 b & \bb\trans \\ 
 \bb &\theta I_3 
\end{pmatrix}.\]
Let $U_A$ be a $4\times 2$ matrix whose columns form a basis of $\ker(A - \theta I)$.  The structure of $A$ implies that $U_A$ has the form 
\[\begin{pmatrix}
 0 & 0 \\
 \bu_1 & \bu_2
\end{pmatrix},\]
 where any linear combination of $\bu_1$ and $\bu_2$ is orthogonal to $\ba$. In particular, $c_1\bu_1 + c_2\bu_2$ must have at least two nonzero entries for any $c_1,c_2\in\mathbb{R}$ unless $c_1=c_2=0$. If $U_A[\mathcal{I},:]$ is singular for some $\mathcal{I}\subseteq\{2,3,4\}$ with $|\mathcal{I}| = 2$, then there must be nonzero $c_1$ and $c_2$ such that $c_1\bu_1 + c_2\bu_2$ has only one nonzero entry, which is impossible.  Therefore, every submatrix $U_A[\mathcal{I},:]$ with $\mathcal{I}\subseteq\{2,3,4\}$ and $|\mathcal{I}| = 2$ must be invertible.  This means $\ker(A - \theta I)$ is ``locally'' generic on $\{2,3,4\}$, and the same behavior happens for $\ker(B - \theta I)$.  

Following the same argument in Theorem~\ref{thm:multi-multi-eigenvalue-in-common}, for any $V_G\subseteq \{2,3,4\}$ with $|V_G| = 2$ and $V_H\subseteq \{6,7,8\}$ with $|V_H| = 3$, the set $\beta = V_G \times V_H$ is a liberation set of $A\oplus B$.  The graph $K_{1,3}\dunion K_{1,3} + \beta$ is shown in Figure~\ref{fig:k13k13}.  Therefore, there is a matrix in $\mptn(K_{1,3}\dunion K_{1,3} + \beta)$ with the SSP and the ordered multiplicity list $(1,1,4,1,1)$, and it is spectrally arbitrary.
\end{example}

\section{Zero forcing}\label{sec:ZF}

In the following, we focus on disconnected graphs and build a technique that depends on  the (classical) zero forcing game to identify SSP liberation sets of corresponding matrices.  The analogous statements for the SAP is included at the end of the section.

We first recall the zero forcing game introduced in \cite{MR2388646}.  Let $G$ be a graph. At each stage of the \emph{zero forcing game}   all vertices are assigned a color: blue or white. At the start of the game the initial set $F \subseteq V(G)$ of blue vertices is chosen. The game is played by repeated application of the following \emph{color change rule}.  If $v$ is the only white neighbor of a blue vertex $u$, then $v$ turns blue in the next step. This action is called \emph{a force} and is denoted by $u\rightarrow v$.    If, starting with an initial set $F$ of blue vertices, repeated application of the color change rule successfully turns all the vertices blue, then $F$ is called a \emph{zero forcing set} for $G$. 

The \emph{zero forcing number} of a graph $G$ is the minimum size of a zero forcing set, denoted by $Z(G)$.  For example, a leaf is a zero forcing set for $P_n$ and $Z(P_n) = 1$, while in $C_n$, any two adjacent vertices form a zero forcing set and $Z(C_n) = 2$.  It is known that $Z(G)$ is an upper bound for the multiplicity of any eigenvalue of any matrix in $\mptn(G)$ \cite{MR2388646}.

\begin{definition}
\label{def:zfcover}
Let $G$ be a graph.  A set $F \subseteq V(G)$ is a \emph{zero forcing cover} of $G$ if $F'$ is a zero forcing set of $G$ for any $F'\subset F$ with $|F'| = |F| - 1$. 
\end{definition}

If $F$ is a zero forcing cover of $G$, then clearly any superset of $F$ is also a zero forcing cover of $G$.  In general, the union of two or more disjoint zero forcing sets is a zero forcing cover.  However, the following example shows that a zero forcing cover does not need to be of this form.

\begin{example}\label{ex:G30}
 It is not difficult to check that the colored sets of vertices of $\mathsf{G_{30}}$ and $\mathsf{G_{36}}$ in Figure~\ref{fig:Paw,S211} are examples of zero forcing covers. More generally, the set of all leaves in a generalized star $G$ is a zero forcing cover of $G$.

 \begin{figure}[h]
\centering
\begin{tikzpicture}
\node[fill=blue] (1) at (-1,0) {};
\node (2) at (0,0) {};
\node (3) at (1,0) {};
\node[fill=blue] (4) at (2,0.5) {};
\node[fill=blue] (5) at (2,-0.5) {};
\draw (1) -- (2) -- (3) -- (4);
\draw (3) -- (5);
\end{tikzpicture}
\hfil
\begin{tikzpicture}
\node[fill=blue] (1) at (-1,0) {};
\node (2) at (0,0) {};
\node (3) at (1,0) {};
\node[fill=blue] (4) at (2,0.5) {};
\node[fill=blue] (5) at (2,-0.5) {};
\draw (1) -- (2) -- (3) -- (4);
\draw (3) -- (5) -- (4);
\end{tikzpicture}
\caption{Zero forcing covers of $\mathsf{G_{30}}$ and $\mathsf{G_{36}}$, colored as blue.}
\label{fig:Paw,S211}
\end{figure}
\end{example}

As we will see below, a zero forcing cover of the Cartesian product of graphs $G$ and $H$ allows us to to find an SSP liberation set of some matrices in $\mptn(G\dunion H)$.  Recall that the \emph{Cartesian product} of two graphs $G$ and $H$ is the graph $G\cart H$ on the vertex set 
\[
 V(G\cart H) = \{(u,v): u\in V(G), v\in V(H)\}
\]
such that two vertices $(u_1,v_1)$ and $(u_2,v_2)$ are adjacent if either 
\begin{itemize}
\item $u_1 = u_2$ and $\{v_1,v_2\}\in E(H)$, or
\item $v_1 = v_2$ and $\{u_1,u_2\}\in E(G)$.
\end{itemize}
It is known \cite{MR2388646} that $Z(P_s\cart P_t) = \min\{s,t\}$ and $Z(C_s\cart P_t) = \min\{s, 2t\}$.  For $s\leq t$, it is known \cite{MR3737108} 
\[Z(C_s\cart C_t) = \begin{cases}
 2s - 1 & \text{if } s = t \text{ and } s \text{ is odd}, \\
 2s & \text{otherwise}.
\end{cases}\]

\begin{lemma}\label{lem:ZF}
Let $G$ and $H$ be graphs, and $\beta'$ a zero forcing set of $G\cart H$. Then $A\oplus B$ has the SSP with respect to $G\dunion H + \beta'$ for any $A\in\mptn(G)$ and $B\in\mptn(H)$ both with the SSP.  
\end{lemma}

\begin{proof}
Let us denote $\Gamma:=G\dunion H + \beta'$ and let $A\in\mptn(G)$, $B\in\mptn(H)$  have the SSP. Moreover, suppose $X\in\mptnclo(\overline{\Gamma})$ satisfies $[A\oplus B, X] = O$. We may write 
\[
X = \begin{pmatrix}
 X_A & Y \\
 Y\trans & X_B
\end{pmatrix}
\]
conformal with the partition of $A\oplus B$ such that $X_A\in\mptnclo(\overline{G})$, $X_B\in\mptnclo(\overline{H})$, and $Y\vert_{\beta'}=0$. Because both $A$ and $B$ have the SSP, the condition $[A\oplus B, X] = O$ is equivalent to $AY - YB= O.$
Next we look at the entries of $AY-YB$, and use the notation  $A = \begin{pmatrix} a_{i,j} \end{pmatrix}$, $B = \begin{pmatrix}  b_{i,j} \end{pmatrix}$, and $Y = \begin{pmatrix} y_{i,j} \end{pmatrix}$. 

Let $(i_0,j_0), (i_1,j_1) \in V(G \cart H)$ and let $F \subset V(G \cart H)$ be a set of blue vertices that allows the force $(i_0,j_0)\rightarrow (i_1,j_1)$ on $G\cart H$ (by one application of the color change rule). 
First we prove that if $y_{i,j} = 0$ for all $(i,j)\in F$ and $(AY - YB)_{i_0,j_0} =0$, then $y_{i_1,j_1} = 0$.  
The equation $(AY-YB)_{i_0,j_0} = 0$ can be written as:
\[\sum_{k \in N_G[i_0]}a_{i_0,k}y_{k,j_0} - \sum_{k \in N_H[j_0]}b_{k,j_0}y_{i_0,k} = 0.\]
From $(i_0,j_0) \rightarrow (i_1,j_1)$, we know that $(i_0,j_0)$ and all its neighbors except $(i_1,j_1)$ in $G\cart H$ are in $F$. This means that all variables $y_{i,j}$ appearing in the equation above  are assumed to be zero except for $y_{i_1,j_1}$. Moreover, since $(i_1,j_1)$ is a neighbor of $(i_0,j_0)$, either $i_1 = i_0$ or $j_1 = j_0$.  If $i_1 = i_0$, then the equation reduces to $-b_{j_1,j_0}y_{i_0,j_1} = 0$.  If $j_1 = j_0$, then it reduces to $a_{i_0,i_1}y_{i_1,j_0} = 0$.  In either case we conclude $y_{i_1,j_1} = 0$.  

Assuming that $Y\vert_{\beta'}=0$ and  $AY-YB=O$, we can now conclude $Y=0$ by repeated application of the claim above. 
Therefore, by Proposition~\ref{prop:directsxp}, $A\oplus B$ has the SSP with respect to $G\dunion H + \beta'$ if $\beta'$ is a zero forcing set of $G\cart H$, as claimed.  
\end{proof}

\begin{theorem}
\label{thm:zf}
Let $G$ and $H$ be graphs.  Suppose $A\in\mptn(G)$ and $B\in\mptn(H)$ have the SSP.  If $\beta$ is a zero forcing cover of $G\cart H$, then $\beta$ is an SSP liberation set of $A\oplus B$.
\end{theorem}
\begin{proof}
If $\beta$ is a zero forcing cover, then any $\beta'\subset\beta$ with $|\beta'| = |\beta| - 1$ is a zero forcing set.  Therefore, by Lemma \ref{lem:ZF} $\beta$ is an SSP liberation set of $A\oplus B$ for any $A\in\mptn(G)$ and $B\in\mptn(H)$ with the SSP.
\end{proof}

In Theorem~\ref{thm:zf} we require that both $A$ and $B$ have the SSP, but we do not assume that $A\oplus B$ has the SSP as well, hence we allow $A$ and $B$ to have some eigenvalues in common.

\begin{example}
\label{ex:pmpn}
Let $2\leq s\leq t$, $G = P_s$ and $H = P_t$.  Let $u_1,\ldots, u_s$ and $v_1,\ldots, v_t$ be the vertices of $G$ and $H$, respectively, following the path order.  Let 
\[\beta = \{\{u_i,v_1\}\colon i\in[s]\} \cup \{\{u_i,v_2\}\colon i\in[s],\ i \equiv 2,3 \hspace{-7pt} \pmod 4\} \cup \{\{u_{s-1}, v_2\}\}.\]
Note that the edge $\{u_{s-1}, v_2\}$ might already be part of the second set in the union above.  Then $\beta$ is a zero forcing cover of $G\cart H$.  By Theorem~\ref{thm:zf}, $\beta$ is an SSP liberation set of $A\oplus B$ for any $A\in\mptn(G)$ and $B\in\mptn(H)$ with the SSP.  Since every matrix of a path has the SSP \cite{MR4080669}, $\beta$ is in fact an SSP liberation set of $G\dunion H$.  Since paths realize any discrete spectrum (spectrum with all eigenvalues distinct), $P_s\dunion P_t + \beta$ realizes any spectrum composed of at most $s$ eigenvalues with multiplicity $2$ and some simple eigenvalues.
\end{example}

Notice that when $H = m K_1$ the Cartesian product $G\cart H$ is isomorphic to the disjoint union of $m$ copies of $G$. Therefore, zero forcing covers  of $G$ can be used to construct a zero forcing cover of $G \cart mK_1$. 

\begin{corollary}
\label{cor:zf}
Let $G$ be a graph on $n$ vertices, $m\in \mathbb{N}$, and $A\in\mptn(G)$ with the SSP.  If $F_j$ are zero forcing covers of $G$, $j\in [m]$, then 
\[\beta = \bigcup_{j\in [m]}\{\{u,i\}\colon u\in F_j,i\in[m]\}\]
is an SSP liberation set of $A\oplus\operatorname{diag}(\lambda_1,\ldots,\lambda_m)$ for any distinct $\lambda_1,\ldots,\lambda_m\in\mathbb{R}$.
\end{corollary}

\begin{example}
Let $P_{n-1}$ have the vertices $u_1,\ldots, u_{n-1}$ following the path order.  Then $\beta = \{u_1,u_{n-1}\}$ is a zero forcing cover of $P_{n-1}$.  By Corollary~\ref{cor:zf}, $\widehat{\beta}=\{\{u_{1},u_n\},\{u_{n-1},u_n\}\}$ is an SSP liberation set of any matrix of the form $A\oplus\begin{pmatrix}\lambda\end{pmatrix}\in\mptn(P_{n-1}\dunion K_1)$, where $A\in\mptn(P_{n-1})$ has the SSP, $\lambda\in\mathbb{R}$ and $V(K_1)=\{u_n\}$. The assumption that $A$ has the SSP can again be ignored by \cite{MR4080669}.  Since $(P_{n-1}\dunion K_1) + \widehat{\beta} \cong C_n$  and any $n - 1$ distinct real numbers can be the spectrum of some $A\in\mptn(P_{n-1})$, Theorem~\ref{lem:liberation} ensures that the unordered multiplicity list $\{2,1,\ldots,1\}$ is spectrally arbitrary with the SSP for $C_n$.  This aligns with \cite[Corollary~7.6]{MR4074182}.
\end{example}

\begin{example}\label{ex:g175}
By \cite[Fig.~1]{MR4074182} the ordered multiplicity list $(1,2,1)$ is spectrally arbitrary and realizable with the SSP for $K_{1,3}$, and by \cite[Lemma~2.2]{MR3034535} there exists a matrix $B\in\mptn(2K_1)$ with distinct eigenvalues $\lambda_1,\lambda_2$ and SSP. Note that the three leaves of $K_{1,3}$ form a zero forcing cover of $K_{1,3}$. Therefore, by Corollary~\ref{cor:zf} and Lemma~\ref{lem:liberation} the ordered multiplicity lists $(1,3,2)$ and $(2,3,1)$ are realizable in $\mathsf{G_{175}}=K_{3,3}$, see Figure~\ref{fig:171,175}, which completes the list of ordered realizable multiplicity lists for $\mathsf{G_{175}}$ in~\cite[Appendix~B]{MR4284782}. 
\end{example}

Let us present some examples in the case when $m=1$ in  Corollary~\ref{cor:zf}.

\begin{example}\label{ex:g129-g145-g153}
 By \cite[Fig.~1]{MR4074182} the ordered multiplicity lists $(1,2,1,1)$ and $(1,1,2,1)$ are spectrally arbitrary and realizable with the SSP for $\mathsf{G_{30}}$, and hence by Corollary~\ref{cor:zf} and Example~\ref{ex:G30} the ordered multiplicity lists $(1,3,1,1)$ and $(1,1,3,1)$ are spectrally arbitrary for $\mathsf{G_{129}}$ with the SSP, see Figure~\ref{fig:g129}.

Since $\mathsf{G_{145}}$ and $\mathsf{G_{153}}$ are supergraphs of $\mathsf{G_{129}}$, it follows by Theorem~\ref{thm:supergraph} that the ordered multiplicity lists $(1,3,1,1)$ and $(1,1,3,1)$ are spectrally arbitrary for $\mathsf{G_{145}}$ and $\mathsf{G_{153}}$ with the SSP. Alternatively, one can arrive at the same conclusion by applying Corollary~\ref{cor:zf} and Example~\ref{ex:G30} to $\mathsf{G_{36}}$ and the zero forcing cover presented in Figure~\ref{fig:Paw,S211}. This gives the answer to spectral arbitrariness of $(1,3,1,1)$ and $(1,1,3,1)$ for $\mathsf{G_{129}}$, $\mathsf{G_{145}}$ and $\mathsf{G_{153}}$ in~\cite[Appendix~B]{MR4284782}. 
\end{example}

\begin{figure}[h]
\centering
\begin{tikzpicture}
\foreach \i in {1,4,6} {
    \pgfmathsetmacro{\angle}{60 * (\i -2)}
    \node[fill=blue,label={\angle:$\i$}] (\i) at (\angle:1) {};
 }
\foreach \i in {2,3} {
    \pgfmathsetmacro{\angle}{60 * (\i -2)}
    \node[label={\angle:$\i$}] (\i) at (\angle:1) {};
 }
\node[fill=black,label={180:$5$}] (5) at (180:1) {};
\draw (1) -- (2) -- (3) -- (4);
\draw (2) -- (6);
\draw[color=blue,thick,dashed] (4) -- (5) -- (6);
\draw[color=blue,thick,dashed] (1) -- (5);
\node[rectangle,draw=none] at (0,-1.5) {$\mathsf{G_{129}}$};
\end{tikzpicture}
\hfil
\begin{tikzpicture}
\foreach \i in {1,4,6} {
    \pgfmathsetmacro{\angle}{60 * (\i -2)}
    \node[fill=blue,label={\angle:$\i$}] (\i) at (\angle:1) {};
 }
\foreach \i in {2,3} {
    \pgfmathsetmacro{\angle}{60 * (\i -2)}
    \node[label={\angle:$\i$}] (\i) at (\angle:1) {};
 }
\node[fill=black,label={180:$5$}] (5) at (180:1) {};
\draw (1) -- (2) -- (3) -- (4);
\draw (2) -- (6);
\draw[color=blue,thick,dashed] (4) -- (5) -- (6);
\draw[color=blue,thick,dashed] (1) -- (5)--(2);
\node[rectangle,draw=none] at (0,-1.5) {$\mathsf{G_{145}}$};
\end{tikzpicture}
\hfil
\begin{tikzpicture}
\foreach \i in {1,4,6} {
    \pgfmathsetmacro{\angle}{60 * (\i -2)}
    \node[fill=blue,label={\angle:$\i$}] (\i) at (\angle:1) {};
 }
\foreach \i in {2,3} {
    \pgfmathsetmacro{\angle}{60 * (\i -2)}
    \node[label={\angle:$\i$}] (\i) at (\angle:1) {};
 }
\node[label={180:$5$},fill=black] (5) at (180:1) {};
\draw (6) -- (1) -- (2) -- (3) -- (4);
\draw (2) -- (6);
\draw[color=blue,thick,dashed] (4) -- (5) -- (6);
\draw[color=blue,thick,dashed] (1) -- (5);
\node[rectangle,draw=none] at (0,-1.5) {$\mathsf{G_{153}}$};
\end{tikzpicture}
\caption{Blue vertices of $\mathsf{G_{30}}$ and $\mathsf{G_{36}}$ present their zero forcing covers. Adding the black vertex
with the label $5$ we obtain graphs $\mathsf{G_{129}} = \mathsf{G_{30}}\cup K_1 + \widehat{\beta}_{30}$ and $\mathsf{G_{153}} = \mathsf{G_{36}}\cup K_1 + \widehat{\beta}_{36}$. Note that $\mathsf{G_{145}}$ is a supergraph of $\mathsf{G_{129}}$.
}
\label{fig:g129}
\end{figure}

\begin{example}\label{ex:g171-g187}
By \cite[Fig.~1]{MR4074182} the ordered multiplicity lists $(1,2,2)$, $(2,2,1)$, $(1,1,2,1)$ and $(1,2,1,1)$ are realizable with the SSP and spectrally arbitrary for $C_5$. Since any four vertices of $C_5$ are a zero forcing cover of $C_5$, using Corollary~\ref{cor:zf} 
it follows that ordered multiplicity lists $(1,2,3)$, $(1,3,2)$, $(3,2,1)$, $(3,1,2)$, $(1,1,3,1)$ and $(1,3,1,1)$ are spectrally arbitrary for $\mathsf{G_{171}}$ with the SSP, see Figure~\ref{fig:171,175}. Moreover, since $\mathsf{G_{187}}$ is a supergraph of $\mathsf{G_{171}}$, the same ordered multiplicity lists are spectrally arbitrary with the SSP for $\mathsf{G_{187}}$ as well. And so we have completely resolved the question of spectral arbitrariness for the remaining ordered multiplicity lists of $\mathsf{G_{171}}$ and $\mathsf{G_{187}}$ listed in \cite[Appendix~B]{MR4284782}.
\end{example}

\begin{figure}[h]
\centering
\begin{tikzpicture}
\foreach \i in {1,3,4,5} {
    \pgfmathsetmacro{\angle}{60 * (\i -1)}
    \node[fill=blue,label={\angle:$\i$}] (\i) at (\angle:1) {};
 }
\node[label=2] (2) at (60:1) {};
\node[fill=black,label={300:$6$}] (6) at (300:1) {};
\draw (1) -- (2) -- (3) -- (4) -- (5) -- (1);
\draw[color=blue,thick,dashed] (1) -- (6) -- (3);
\draw[color=blue,thick,dashed] (4) -- (6)--(5);
\node[rectangle,draw=none] at (0,-1.5) {$\mathsf{G_{171}}$};
\end{tikzpicture}
\hfil
\begin{tikzpicture}
\foreach \i in {1,3,5} {
    \pgfmathsetmacro{\angle}{60 * (\i -1)}
    \node[fill=blue,label={\angle:$\i$}] (\i) at (\angle:1) {};
 }
\foreach \i in {2,4} {
    \pgfmathsetmacro{\angle}{60 * (\i -1)}
    \node[fill=black,label={\angle:$\i$}] (\i) at (\angle:1) {};
 }
\node[label={300:$6$}] (6) at (300:1) {};
\draw (1) -- (6) -- (3);
\draw (5) -- (6);
\draw[color=blue,thick,dashed] (5) -- (4) -- (1);
\draw[color=blue,thick,dashed] (5) -- (2) -- (1);
\draw[color=blue,thick,dashed] (4) -- (3) -- (2);
\node[rectangle,draw=none] at (0,-1.5) {$\mathsf{G_{175}}$};
\end{tikzpicture} \hfil
\begin{tikzpicture}
\foreach \i in {1,2,3,4,5} {
    \pgfmathsetmacro{\angle}{60 * (\i -1)}
    \node[fill=blue,label={\angle:$\i$}] (\i) at (\angle:1) {};
 }
\node[fill=black,label={300:$6$}] (6) at (300:1) {};
\draw (1) -- (2) -- (3) -- (4)--(5)--(1);
\draw[color=blue,thick,dashed] (1) -- (6) -- (3);
\draw[color=blue,thick,dashed] (4) -- (6)--(5);
\draw[color=blue,thick,dashed] (6)--(2);
\node[rectangle,draw=none] at (0,-1.5) {$\mathsf{G_{187}}$};
\end{tikzpicture}
\caption{Blue vertices of $C_5$ and $K_{1,3}$ present their zero forcing covers. Adding black $K_1$ and $2 K_1$ to $C_5$ and $K_{1,3}$, respectively, we obtain graphs $\mathsf{G_{171}}$, $\mathsf{G_{175}}$   $\mathsf{G_{187}}$.}
\label{fig:171,175}
\end{figure}

While we were able to use the classical zero forcing on the Cartesian product of graphs to build the SSP liberation set, we need to define a new color change rule for the SAP.  Let $G$ and $H$ be graphs, and let each vertex of $G\cart H$ be colored blue or white.  We say $u\xrightarrow{G} v$ if $u$ and $v$ are in the same copy of $G$ and by only looking at this induced subgraph isomorphic to $G$ the action $u\rightarrow v$ is allowed by the color change rule.  The notion of $u\xrightarrow{H} v$ is defined similarly. If one may start by coloring a set of vertices $F$ blue and repeatedly apply $u\xrightarrow{G} v$ or $u\xrightarrow{H} v$ to make $V(G\cart H)$ blue, then $F$ is called a \emph{local zero forcing set} of $G\cart H$.  

\begin{definition}
\label{def:lzfcover}
Let $G$ be a graph.  A set $\beta \subseteq V(G)$ is a \emph{local zero forcing cover} of $G$ if $\beta'$ is a local zero forcing set of $G$ for any $\beta'\subset\beta$ with $|\beta'| = |\beta| - 1$.
\end{definition}

By modifying the proof of Theorem~\ref{thm:zf}, we obtain the analogous result for the SAP.  

\begin{theorem}
\label{thm:zfsap}
Let $G$ and $H$ be graphs.  Suppose $A\in\mptn(G)$ and $B\in\mptn(G)$ have the SAP.  If $\beta$ is a local zero forcing cover of $G\cart H$, then $\beta$ is an SAP liberation set of $A\oplus B$.
\end{theorem}

\begin{remark}
Note that when $H = K_1$, a set $F$ is a zero forcing cover of $G\cart H$ if and only if $F$ is a local zero forcing cover of $G\cart H$.  This is reasonable, since the SSP liberation set allows one to add an arbitrary eigenvalue $\lambda$ on top of $\spec(A)$, while the SAP liberation set allows one to increase the multiplicity of an eigenvalue of $\spec(A)$ --- they have the same effect. 
\end{remark}

\begin{example}
\label{ex:prism}
Let $G = C_4$ and $H = P_2$, and let $V(G) = [4]$ and $V(H) = \{5,6\}$.  Any set of the form $\{(i,5), (i+1,5)\}$ or $\{(i,6), (i+1,6)\}$ is a local zero forcing set for $G\cart H$, where $i+1$ is replaced by $1$ if $i = 4$.  Consequently, 
\[
F = \{(1,5), (2,5), (3,6), (4,6)\}
\]
is a local zero forcing cover for $G\cart H$.  Let $\beta$ be the edge set corresponding to $F$.  Then by Theorem~\ref{thm:zfsap}, $\beta$ is a liberation set of $A\oplus B$ for any $A\in\mptn(G)$ and $B\in\mptn(H)$ with the SAP.  By choosing $A$ with nullity $2$ and $B$ with nullity $1$, we prove the existence of a matrix in  $\mptn(G\dunion H + \beta)$ with nullity $3$ and the SAP. Note that $G\dunion H + \beta$ is the prism graph shown in Figure~\ref{fig:prism}.

\begin{figure}[h]
\centering
\begin{tikzpicture}
\node[label={left:$1$}] (1) at (0,3) {};
\node[label={left:$2$}] (2) at (1,2) {};
\node[label={left:$3$}] (3) at (1,0) {};
\node[label={left:$4$}] (4) at (0,1) {};
\node[label={right:$5$}] (5) at (2,2.8) {};
\node[label={right:$6$}] (6) at (2,0.8) {};

\draw (1) -- (2) -- (3) -- (4) -- (1);
\draw (5) -- (6);
\draw[blue,thick,dashed] (1) -- (5) -- (2);
\draw[blue,thick,dashed] (3) -- (6) -- (4);
\end{tikzpicture}
\caption{The prism graph $C_4\dunion P_2 + \beta$ for Example~\ref{ex:prism}.}
\label{fig:prism}
\end{figure}
\end{example}

\section{Liberation set of a graph}
\label{sec:libergraph}

In this section, we study the sets that are liberation sets for any matrix in $\mptn(G)$, and hence do not depend on the choice of the matrix in $\mptn(G)$. In general, it can be hard to identify a liberation set of a graph.  However, the SSP sequence introduced in \cite{MR4080669} and the SAP zero forcing introduced in \cite{MR3536955} provide combinatorial tools that can help.

\begin{definition}
Let $G$ be a graph.  
A subset $\beta\in E(\overline{G})$ is an \emph{SSP liberation set of $G$} (an \emph{SAP liberation set of $G$}, respectively)  if $\beta$ is an SSP (an SAP, respectively) liberation set of $A$ for all $A\in\mptn(G)$. 
\end{definition}

As an immediate corollary of Lemma~\ref{lem:liberation} we obtain the following result.
\begin{theorem}
\label{thm:gliberation}
Let $G$ be a graph and $\beta$ an SSP (an SAP, respectively) liberation set of $G$.  Then any spectrum (or rank, respectively) realizable in $\mptn(G)$ is realizable in $\mptn(G + \beta)$ with the SSP (the SAP, respectively).
\end{theorem}

\begin{example}
Consider $G$ to be  a star graph $K_{1,4}$ with $v=5$ as the only non-leaf vertex, and let
\[A = \begin{pmatrix}
 d_1 & 0 & 0 & 0 & a_{1,5} \\
 0 & d_2 & 0 & 0 & a_{2,5} \\
 0 & 0 & d_3 & 0 & a_{3,5} \\
 0 & 0 & 0 & d_4 & a_{4,5} \\
 a_{1,5} & a_{2,5} & a_{3,5} & a_{4,5} & d_5
\end{pmatrix}\in\mptn(K_{1,4}),\]
where $a_{i,j}$'s are assumed to be nonzero while $d_i$'s can be any real numbers.
The SSP verification matrix $\Psi$ of $A$ is equal to
\[
\scalemath{0.8}{
\begin{pmatrix}
d_{1} - d_{2} & 0 & 0 & -a_{2,5} & 0 & 0 & -a_{1,5} & 0 & 0 & 0 \\
0 & d_{1} - d_{3} & 0 & -a_{3,5} & 0 & 0 & 0 & 0 & -a_{1,5} & 0 \\
0 & 0 & d_{1} - d_{4} & -a_{4,5} & 0 & 0 & 0 & 0 & 0 & -a_{1,5} \\
0 & 0 & 0 & 0 & d_{2} - d_{3} & 0 & -a_{3,5} & 0 & -a_{2,5} & 0 \\
0 & 0 & 0 & 0 & 0 & d_{2} - d_{4} & -a_{4,5} & 0 & 0 & -a_{2,5} \\
0 & 0 & 0 & 0 & 0 & 0 & 0 & d_{3} - d_{4} & -a_{4,5} & -a_{3,5}
\end{pmatrix}},
\]
where the rows are indexed by the nonedges $\{1,2\}$, $\{1,3\}$, $\{1,4\}$, $\{2,3\}$, $\{2,4\}$, $\{3,4\}$.  Let $\beta' = \{\{1,2\},\{1,3\}\}$ and  $\gamma$ be the set of columns that contain no entry of the form $d_i - d_j$, $i\ne j$.  Thus, 
\[\Psi:=\Psi[E(\overline{G})\setminus\beta',\gamma] = \begin{pmatrix}
 -a_{4,5} & 0 & 0 & -a_{1,5} \\
 0 & -a_{3,5} & -a_{2,5} & 0 \\
 0 & -a_{4,5} & 0 & -a_{2,5} \\
 0 & 0 & -a_{4,5} & -a_{3,5}
\end{pmatrix}.\]
As $\det(\Psi)=-2a_{2,5}a_{3,5}a_{4,5}^2$, it follows that $\Psi$ has full row-rank for any choice of nonzero $a_{i,j}$'s. By Remark~\ref{rem:vermtx}, $A$ has the SSP with respect to $K_{1,4} + \beta'$.  By symmetry, the same argument works for any $\beta'$ of the form $\{\{i,j\},\{i,k\}\}$ with $i,j,k\in [4]$.  Therefore, both 
\[\begin{aligned} 
\beta_1 &= \{\{1,2\}, \{2,3\}, \{1,3\}\}, \text{ and}\\
\beta_2 &= \{\{1,2\}, \{1,3\}, \{1,4\}\}
\end{aligned}\]
are SSP liberation sets of $K_{1,4}$, as illustrated in Figure~\ref{fig:k14liber}.

\begin{figure}[h]
\centering
\begin{tikzpicture}
\node[label={left:$1$}] (1) at (0,2) {};
\node[label={right:$2$}] (2) at (2,2) {};
\node[label={left:$3$}] (3) at (0,0) {};
\node[label={right:$4$}] (4) at (2,0) {};
\node[label={left:$5$}] (5) at (1,1) {};
\draw (1) -- (5) -- (2);
\draw (3) -- (5) -- (4);
\draw[blue, thick,dashed] (1) -- (2);
\draw[blue, thick,dashed] (2) to[bend left] (3);
\draw[blue, thick,dashed] (3) -- (1);
\end{tikzpicture}
\hfil
\begin{tikzpicture}
\node[label={left:$1$}] (1) at (0,2) {};
\node[label={right:$2$}] (2) at (2,2) {};
\node[label={left:$3$}] (3) at (0,0) {};
\node[label={right:$4$}] (4) at (2,0) {};
\node[label={left:$5$}] (5) at (1,1) {};
\draw (1) -- (5) -- (2);
\draw (3) -- (5) -- (4);
\draw[blue, thick,dashed] (1) -- (2);
\draw[blue, thick,dashed] (1) -- (3);
\draw[blue, thick,dashed] (1) to[bend left] (4);
\end{tikzpicture}
\caption{Two SSP liberation sets of $K_{1,4}$.}
\label{fig:k14liber}
\end{figure}

As a consequence, any spectrum occurring in $\mptn(K_{1,4})$ also occurs in $\mptn(K_{1,4} + \beta_1)$ and $\mptn(K_{1,4} + \beta_2)$ with the SSP by Theorem~\ref{thm:gliberation}. On the other hand, it is known that not all $A\in\mptn(K_{1,4})$ have the SSP.  In particular, no matrix $A\in\mptn(K_{1,4})$ with an eigenvalue of multiplicity $3$ has the SSP; see, e.g., \cite{MR4074182}.  
\end{example}

The theory developed above is useful for studying various parameters that depend on the spectrum of matrices. Two examples of this are specified below.  
Recall that $M(G)$ is the maximum nullity of matrices in $\mptn(G)$, and $\xi(G)$ denotes the maximum nullity of matrices in $\mptn(G)$ with the SAP~\cite{MR2181887}. Moreover, let $q(G)$ denote the minimal number of distinct eigenvalues of matrices in $\mptn(G)$ and  $q_S(G)$ the minimal number of distinct eigenvalues of matrices in $\mptn(G)$ with SSP~\cite{MR3665573}. 
By Theorem~\ref{thm:gliberation} we get the following inequalities.

\begin{corollary}\label{cor:liberation-set-graph}
Let $G$ be a graph.  
\begin{enumerate}
    \item If $\beta$ an SSP liberation set of $G$, then $q(G + \beta)\leq q_S(G + \beta) \leq q(G)\leq q_S(G)$.
    \item If $\beta$ an SAP liberation set, then $M(G + \beta)\geq \xi(G + \beta)\geq M(G) \geq \xi(G)$.
\end{enumerate}
\end{corollary}

\section{Conclusion}
\label{sec:conclusion}

The introduction of the strong spectral properties has made a powerful impact on the IEP-$G$ and related problems. The notion of the SSP liberation set, introduced in this paper, further advances the approach to the problem through an investigation of what perturbations of the pattern of a given matrix guarantee the preservation of spectra or rank. While the paper offers a selection of examples where the SSP liberation set is applied, it does not aim to provide an exhaustive list of possible research directions that could benefit from it. A study of the direct sums of matrices and disjoint unions of graphs is a natural first step, especially since using the standard spectral properties in this context is fully understood. Even in this special case, the paper provides a foundation for future research supported with only initial examples that can be developed further. In particular, the paper introduces and motivates a purely combinatorial problem of determining the zero forcing covers of Cartesian products of graphs.

\section*{Acknowledgements}

Jephian C.-H. Lin was supported by the Young
Scholar Fellowship Program (grant no.~NSTC-111-2628-M-110-002) from the National Science and Technology Council of Taiwan.  
Polona Oblak received funding from Slovenian Research Agency (research core funding no.~P1-0222 and project no.~J1-3004).

\bibliographystyle{plain}
\bibliography{references}

\end{document}